\documentclass[12pt]{amsart}
\usepackage{amssymb}
\usepackage{amsfonts}
\usepackage{textcomp}
\usepackage[centertags]{amsmath}
\usepackage{latexsym}
\usepackage{amsthm}
\usepackage{indentfirst}
\usepackage{fancyhdr}
\usepackage[english]{babel}
\usepackage[english]{babel}
\usepackage{color}
\usepackage{hyperref}
\usepackage[utf8]{inputenc}
\usepackage[T1]{fontenc}
\usepackage{esint}

\theoremstyle{plain}
\newtheorem{thm}{Theorem}[section]
\newtheorem{Cor}{Corollaire}[section]
\newtheorem{Def}{Definition}[section]
\newtheorem{rem}{Remark}[section]
\newtheorem{prop}{Proposition}[section]
\newtheorem{Lem}{Lemma}[section]

\usepackage{t1enc}

\begin{document}
	
	%\begin{frontmatter}
	
	\title[]{\bf On the existence and regularity of an optimal shape for the non-linear first eigenvalue problem with Dirichlet condition}
	\author[R. M. Gouton]{Rocard Michel GOUTON$^1$}
	\author[A. Marcos]{Aboubacar MARCOS$^1$}
	\author[D. Seck]{Diaraf SECK$^2$}
	\address[1]{Institut de Math\'ematiques et de Sciences Physiques, Porto Novo, Benin }
	\address[2]{Universit\'e CHeikh Anta Diop de Dakar}
	\maketitle	\begin{abstract}
		We study a shape optimization problem associated with the first eigenvalue of a nonlinear spectral problem involving a mixed operator ($p-$Laplacian and Laplacian) with a constraint on the volume. First, we prove the existence of a quasi-open $\Omega^*\subset D$ minimizer of the first eigenvalue under a volume constraint. Next, the local continuity of the eigenfunction associated with the eigenvalue on $\Omega^*$ is proved. This allows us to conclude that $\Omega^*$ is open when $D$ is connected. This is an important first step for regularizing the optimal shape themselves. Finally, there is a proof that the reduced boundary of the optimal shape is regular.
		%		\textbf{Mathematics Subject Classification }:  XXXX, XXXX.
	\end{abstract}

	\section{Introduction}
	Let $\Omega$ be an open subset of a fixed ball $D$ in $\mathbb{R}^{d} $ $d\geq2$. We consider the following nonlinear eigenvalue problem
	\begin{equation}\label{valeurchar2}
		\left
		\{
		\begin{array}{c r}
			
			-\Delta_p u -\Delta 
			u   = \lambda |u|^{q-2}u & \mbox{dans} \ \Omega \\
			%	u>0 &  \mbox{dans} \ \Omega \\
			u  =  0  & \mbox{sur}  \ \partial\Omega
		\end{array}
		\right.
	\end{equation}
	where $\Delta_p u  :=div(\arrowvert \nabla u\arrowvert ^{p-2}\nabla u)$ is the $p-$Laplace operator and $p> q\geq2$.
	We say that $\lambda\in \mathbb{R}$ is an eigenvalue of the problem (\ref{valeurchar2}) if there exists $u_\lambda\in W_0^{1,p}(\Omega)\smallsetminus \{ 0\} $ such that
	\begin{equation}\label{vari12char4}
		\int_{\Omega}(\arrowvert \nabla u\arrowvert ^{p-2}+1)\nabla u_\lambda \nabla\varphi dx-\lambda\int_{\Omega} |u_\lambda|^{p-2}u_\lambda\varphi dx=0 \ \ \ \forall \varphi \in W_0^{1,p}(\Omega)\smallsetminus \{ 0\}.
	\end{equation} 
	Such a function $u_{\lambda}$ will be called an eigenfunction corresponding to the eigenvalue $\lambda$.
	Contrary to the Laplace operator, it is proved in \cite{Rocard1} that the spectrum is continuous. We denote by $\lambda_{1}(\Omega)$ the first eigenvalue of the problem (\ref{valeurchar2}).
	We are interested in optimization problems of the shape of the minimization of the eigenvalue $\lambda_1(\Omega)$ under different types of constraints, i.e. the following minimization problems	
	\begin{equation}\label{minpro1}
		\min\{\lambda_{1}(\Omega),\Omega\subset D ,\ \Omega \  \mbox{quasi-open }, |\Omega|=c \}
	\end{equation}
	\begin{equation}\label{}
		\min\{\lambda_{1}(\Omega), \ \Omega\subset D, |\Omega|=c \}
	\end{equation}
	and
	\begin{equation}\label{minpro2}
		\min\{ \lambda_{1}(\Omega),\Omega\in \mathcal{A}, |\Omega|\leq c\}
	\end{equation}
	where $\mathcal{A}$ is a family of admissible domain defined by
	\begin{equation*}
		\mathcal{A}=\{ \Omega\subseteq D, \Omega\ \mbox{is quasi-open} \}	.
	\end{equation*}
	Note that in a previous work we have shown the isoperimetric inequality of the Rayleigh-Faber Khran type, i.e. the ball minimizes the eigenvalue $\lambda_1(\Omega)$. In other words, we have 	\begin{equation*}
		\lambda_{1}(B)=\min\{\lambda_{1}(\Omega),\Omega \ \mbox{open set of} \ \mathbb{R}^d, |\Omega|=c \}	
	\end{equation*}
	where $B$ is a ball of $\mathbb{R}^d$ of fixed volume.\\
	Various questions have been studied for these types of problems. The existence of the solution has been studied for the first values of the Laplacian and for the $p-$Laplacian (cf. \cite{IDRISSA,Henrot2017}). As for the question of the regularity of the optimal shape, it has been studied only in the case of the Laplacian. The study of the regularity of the optimal shape of these types of shape optimization problems started in 1981 with the paper of L.A. Caffarelli et al. \cite{Caffarelli1}, where they show that the optimal shape is an open one and the singular part of its boundary is of null Haussdorf measure. Later, in 1996, B.Gustasson et al. \cite{Gustafsson11} used the same techniques to show the same kind of results for the optimal solution of another shape optimization problem. Subsequently, several other related works were published \cite{Briancon,Briancon189,Briancon3}. More recently, in 2017, the book by Antoine Henrot (ed.) \cite{Henrot2017} reviews various issues on shape optimization problems of type (\ref{minpro1}) and (\ref{minpro2}).\\
	The goal of this article is first to study the existence of a solution to the problem (\ref{minpro2}). In a second step, we study the regularity of the optimal solution. The study of regularity is divided into two parts, in the first we study the Lipschitz continuity results for the solution of the state and in the second we study the regularity of the boundary of the optimal solution. This work is an attempt to extend the results obtained by T. Briançon, M. Hayouni, M. Pierre and J. Lamboley in \cite{Briancon,Briancon3,Briancon189,Henrot2017}.\\
	The rest of the paper is organized as follows. In section 2, we give the definition of the eigenvalue $\lambda_1(\Omega)$ and recall some properties of its associated eigenfunctions. In the section 3, we study the properties of geometric variations for the first eigenvalue. In section 4, we show the existence of a solution for the problem (\ref{minpro2}). In Section 5, we study the Lipschitz continuity results for the solution of the state. Finally in Section 6, we show the regularity of the boundary of the optimal shape.
	\section{Continuity result}
Let $\Omega \subset	\mathbb{R}^d, d \geq2$ an open bounded set, the first eigenvalue of the operator $\Delta_p + \Delta$ with Dirichlet condition is defined by the nonlinear Rayleigh quotient
	\begin{equation}
		\lambda_{1}(\Omega):=\inf_{u\in W_0^{1,p}(\Omega)\smallsetminus \{ 0\}}\dfrac{\frac{1}{p}\int_{\Omega}\arrowvert \nabla u\arrowvert ^{p}dx+\frac{1}{2}\int_{\Omega}\arrowvert \nabla u\arrowvert ^{2}dx}{\frac{1}{q}\int_{\Omega}|u|^{q}}.
	\end{equation}
	The first eigenvalue is strictly positive. Let $u_{\Omega}$ be a positive eigenfunction associated with $\lambda_1(\Omega)$ , we have $u_{\Omega}\in W_0^{1,p}(\Omega)\cap L^{\infty}(\Omega)$ see \cite{Rocard1}.
	In this section we are interested in the continuity of the map
	\begin{equation*}
		\Omega\longmapsto \lambda_1(\Omega).
	\end{equation*}
	Then, we have to fix topology on the space of the open subsets of $D$ where $D \subset \mathbb{R}^d$ is a bounded set . On the family of the open
	subsets of $D$, we define the Hausdorff complementary topology, denoted $H^c$ given by the metric
	\begin{equation*}
		d_{H^c}(\Omega_1,\Omega_2)=\sup_{x\in\mathbb{R}^d}|d(x,\Omega_1^c)-d(x,\Omega_2^c)|.	
	\end{equation*}
	The $H^c-$topology has some good properties for example the space of the open subsets of $ D$ is
	compact. Moreover if $\Omega_n  \ \xrightarrow{H^c} \ \Omega$, then for any compact  $K\subset\subset \Omega$  we have $K\subset\subset\Omega_n$ for $n$ large
	enough.\\
	However, perturbations in this topology may be very irregular and in general situations the
	continuity of the mapping $ \Omega\longmapsto \lambda_1(\Omega)$ fails.\\
	In order to obtain a compactness result we impose some additional constraints on the space
	of the open subsets of $D$ which are expressed in terms of the Sobolev capacity. There are many
	ways to define the Sobolev capacity, we use the local capacity defined in the following way.
	
	\begin{Def}
		For a compact set $K$ contained in a ball $D$,
		\begin{equation*}
			cap(K,D):=\inf \left\lbrace \int_{D}\arrowvert \nabla \varphi\arrowvert ^{p} dx \ \ \  \varphi \in C^\infty_c(B)  \ and \  \varphi\geq 1 on  \ K\ \right\rbrace.
		\end{equation*}
	\end{Def}
	
	\begin{Def}
		\begin{enumerate}
			\item It is said that a property holds $p-$quasi everywhere (abbreviated as $p- q.e$) if it holds
			outside a set of $p-$capacity zero.
			\item A set $\Omega\subset \mathbb{R}^{N}$ is said to be quasi open if for every $\epsilon> 0$ there exists an open set $\Omega_\epsilon$  such that $\Omega\subseteq \Omega_\epsilon$ and $cap(\Omega_\epsilon\backslash \Omega) <\epsilon.$ 
			\item A function $u:\mathbb{R}^d\longrightarrow\mathbb{R}$ is said $p-$quasi continuous if for every $\epsilon>0$ here exists an
			open set $\Omega_\epsilon$ such that $cap(\Omega_\epsilon) <\epsilon$ and $u_{|\mathbb{R}^d \backslash \Omega_\epsilon}$ is continuous in $\mathbb{R}^d\backslash\Omega_\epsilon$
		\end{enumerate}
	\end{Def}
	It is well known that every Sobolev function $u\in W^{1,p}(\mathbb{R}^d)$ has a p-quasi continuous representative which we still denote $u$. Therefore, level sets of Sobolev functions are $p-$quasi open
	sets; in particular $\Omega_u=\{x\in D; |u(x)|>0\}$ is quasi open subsets of $D$.
	
	\begin{Def}
		We say that an open set $\Omega$ has the $p-(r,c)$ capacity density condition if
		\begin{equation*}
			\forall x\in \partial\Omega, \ \ \forall 0<\delta<r, \ \dfrac{cap(\Omega^c\cap\bar{B}(x,\delta),B(x,2\delta))}{cap(\bar{B}(x,\delta),B(x,2\delta))} \geq c
		\end{equation*}
		where $B(x,\delta)$ denotes the ball of raduis $\delta$, centred at $x$.
	\end{Def}
	\begin{Def}
		We say that the sequence of the spaces $W_0^{1,p}(\Omega_n)$ converges in the sense of
		Mosco to the space  $W_0^{1,p}(\Omega_n)$ if the following conditions hold
		\begin{enumerate}
			\item[(1)] The first Mosco condition: For all $\phi \in W_0^{1,p} (\Omega)$ there exists a sequence $\phi_n \in W_0^{1,p}(\Omega_n)$ such that $\phi_n$ converges strongly in $W_0^{1,p}(D)$ to $\phi$.
			\item[(2)]	The second Mosco condition: For every sequence $\phi_{nk} \in W_0^{1,p}(\Omega_{nk})$ weakly convergent
			in $W_0^{1,p}(D)$ to a function $\phi$, we have $\phi \in W_0^{1,p}(\Omega).$
		\end{enumerate}
	\end{Def}
	\begin{Def}
		We say a sequence $(\Omega_n)$ of open subsets of a fixed ball $D$ $\gamma_p$-converges to $\Omega$ if
		for any $f \in W^{-1,q}(\Omega)$ the solutions of the Dirichlet problem
		\begin{equation*}
			-\Delta_p u_n-\Delta u_n=f \ dans \ \Omega_n, u_n\in W_0^{1,p}(\Omega_n)
		\end{equation*}
		converge strongly in $W_0^{1,p}(D)$, as $n\longrightarrow +\infty$ to the solution of the corresponding problem in $\Omega$, see \cite{GDAL1},\cite{GDAL2}.
	\end{Def}
	
	By $\mathcal{O}_{p-(r,c)}(D)$,we denote the family of all open subsets of $D$ which satisfy the  
	$p-(r,c)$ capacity density condition.This family is compact in the $H^c$ topology see \cite{GDAL4}. In \cite{GDAL5}, D. Bucur and
	P. Trebeschi, using capacity constraints analogous to those introduce in \cite{GDAL3} and \cite{GDAL4} for the
	linear case, prove the $\gamma_p-$compactness result for the $p-$Laplacian. In the same way, they extend
	the continuity result of Šveràk \cite{GDAL6} to the p-Laplacian for $p\in (d-1,d], d\geq3 .$  The reason
	of the choice of $p$ is that in $\mathbb{R}^d$ the curves have $p$ positive capacity if $p>d-1 .$ The case $p > d$ is trivial since all functions in
	$W^{1,p}(\mathbb{R}^d)$ are continuous.
	Let us denote by 
	$$\mathcal{O}_l(D)=\{\Omega\subseteq D,\sharp\Omega^c\leq l \}$$
	where $\sharp$ denotes the number of the connected components and $l$ a given integer. We have the following theorem.
	\begin{thm}[Bucur-Trebeschi]\label{Bucur}
		Let $d\geq p>d-1.$  Let $\Omega_n$ a sequence
		of open subsets of $D$ such that $(\Omega_n)\subseteq \mathcal{O}_l(D)$ and assume that $\Omega_n$ converges in the Hausdorff complementary topology to $\Omega$. Then $\Omega \in \mathcal{O}_l(D)$ and $\Omega_n$ $\gamma_p-$converges to $\Omega.$
	\end{thm}
	\begin{proof}[Preuve du Théorème \ref{Bucur}]
		voir \cite{GDAL5}
	\end{proof}
	
	For $d = 2$ and $ p = 2$,theorem \ref{Bucur} becomes the continuity result of  Šveràk \cite{GDAL6}.
	Back to the continuity result, we use the above results to prove the following theorem.	
	\begin{thm}\label{conver}
		Let $p>d-1.$ Consider the sequence $(\Omega_n)\subseteq \mathcal{O}_l(D).$ Assume that $\Omega_n$ converges in Hausdorff
		complementary topology to $\Omega$. Then $\lambda_{1}(\Omega_n)$ converges to $\lambda_{1}(\Omega).$
		
	\end{thm}
	
	\begin{proof}[Proof of Theorem\ref{conver}]
		The case $p > d$ is trivial since all functions in
		$W^{1,p}(\mathbb{R}^d)$ are continuous. Let $d\geq p > d-1.$
		Let us take
		\begin{equation*}
			\lambda_{1}(\Omega_n):=\inf_{\phi_n\in W_0^{1,p}(\Omega_n)\smallsetminus \{0\}}\dfrac{\frac{1}{p}\int_{\Omega_n}\arrowvert \nabla \phi_n\arrowvert ^{p}dx+\frac{1}{2}\int_{\Omega_n}\arrowvert \nabla \phi_n\arrowvert ^{2}dx}{\frac{1}{q}\int_{\Omega_n}|\phi_n|^{q}}=\dfrac{\frac{1}{p}\int_{\Omega_n}\arrowvert \nabla u_n\arrowvert ^{p}dx+\frac{1}{2}\int_{\Omega_n}\arrowvert \nabla u_n\arrowvert ^{2}dx}{\frac{1}{q}\int_{\Omega_n}|u_n|^{q}}
		\end{equation*}
		where the minimum is attained by $u_n$, and
		\begin{equation*}
			\lambda_{1}(\Omega):=\inf_{\phi\in W_0^{1,p}(\Omega)\smallsetminus \{0\}}\dfrac{\frac{1}{p}\int_{\Omega}\arrowvert \nabla \phi\arrowvert ^{p}dx+\frac{1}{2}\int_{\Omega}\arrowvert \nabla \phi_n\arrowvert ^{2}dx}{\frac{1}{q}\int_{\Omega}|\phi|^{q}}=\dfrac{\frac{1}{p}\int_{\Omega}\arrowvert \nabla u_1\arrowvert ^{p}dx+\frac{1}{2}\int_{\Omega}\arrowvert \nabla u_1\arrowvert ^{2}dx}{\frac{1}{q}\int_{\Omega}|u_1|^{q}}
		\end{equation*}
		where the minimum is attained by $u_1.$\\
		By the Bucur and Trebeschi theorem, $\Omega_n$ $\gamma_p-$converges to $\Omega$.This implies $W^{1,p}(\Omega_n)$ converges in the sense of Mosco to $W^{1,p}(\Omega).$\\
		If the sequence $(u_n)$ is bounded in $W_0^{1,p}(D)$ then there exists a subsequence still denoted
		$u_n$ such that $u_n$ converges weakly in $W^{1,p}_0(D)$ to a function $u$. The second condition of Mosco
		implies that $u\in W^{1,p}(\Omega).$
		Using the weak lower semicontinuity of the $L^{p}-$norm, we have the inequality
		\begin{equation*}
			\lim\limits_{n\rightarrow +\infty}\inf\dfrac{\frac{1}{p}\int_{D}\arrowvert \nabla u_n\arrowvert ^{p}dx+\frac{1}{2}\int_{D}\arrowvert \nabla u_n\arrowvert ^{2}dx}{\frac{1}{q}\int_{D}|u_n|^{q}dx}\geq\dfrac{\frac{1}{p}\int_{\Omega}\arrowvert \nabla u\arrowvert ^{p}dx+\frac{1}{2}\int_{\Omega}\arrowvert \nabla u\arrowvert ^{2}dx}{\frac{1}{q}\int_{\Omega}|u|^{q}dx}\geq\dfrac{\frac{1}{p}\int_{\Omega}\arrowvert \nabla u_1\arrowvert ^{p}dx+\frac{1}{2}\int_{\Omega}\arrowvert \nabla u_1\arrowvert ^{2}dx}{\frac{1}{q}\int_{\Omega}|u_1|^{q}dx}?
		\end{equation*}
		then  
		\begin{equation}\label{to1}
			\lim\limits_{n\rightarrow +\infty}\inf \lambda_{1}(\Omega_n)\geq \lambda_{1}(\Omega).
		\end{equation}
		Using the first condition of Mosco, there exists a sequence $(v_n)\in W_0^{1,p}(\Omega_n)$ such that $v_n$  converges strongly in $W_0^{1,p}(D)$ to $u_1$.
		
		We have
		\begin{equation*}
			\lambda_{1}(\Omega_n)\leq\dfrac{\frac{1}{p}\int_{\Omega}\arrowvert \nabla u_n\arrowvert ^{p}dx+\frac{1}{2}\int_{\Omega}\arrowvert \nabla u_n\arrowvert ^{2}dx}{\frac{1}{q}\int_{\Omega}|u_n|^{q}}	
		\end{equation*}
		this implies that
		
		\begin{eqnarray*}
			\lim\limits_{n\rightarrow +\infty}\sup \lambda_{1}(\Omega_n)&\leq& \lim\limits_{n\rightarrow +\infty}\sup\dfrac{\frac{1}{p}\int_{D}\arrowvert \nabla u_n\arrowvert ^{p}dx+\frac{1}{2}\int_{D}\arrowvert \nabla u_n\arrowvert ^{2}dx}{\frac{1}{q}\int_{D}|u_n|^{q}}\\ \nonumber
			&= &\lim\limits_{n\rightarrow +\infty}\dfrac{\frac{1}{p}\int_{D}\arrowvert \nabla u_n\arrowvert ^{p}dx+\frac{1}{2}\int_{\Omega}\arrowvert \nabla u_n\arrowvert ^{2}dx}{\frac{1}{q}\int_{D}|u_n|^{q}}\\ \nonumber
			&=& \dfrac{\frac{1}{p}\int_{\Omega}\arrowvert \nabla u_1\arrowvert ^{p}dx+\frac{1}{2}\int_{\Omega}\arrowvert \nabla u_1\arrowvert ^{2}dx}{\frac{1}{q}\int_{\Omega}|u_1|^{q}}
		\end{eqnarray*}
		then 
		\begin{eqnarray}\label{to2}
			\lim\limits_{n\rightarrow +\infty}\sup \lambda_{1}(\Omega_n)&\leq&\lambda_{1}(\Omega)
		\end{eqnarray}
		By the relations (\ref{to1}) and (\ref{to2}) we conclude that $\lambda_{1}(\Omega_n)$ converges to $\lambda_{1}(\Omega)$. 
	\end{proof}	
	\section{Existence results}
	%	We extend the classical inequality of Faber-Krahn for the first eigenvalue of the Dirichlet
	%	Laplacian to the Dirichlet $p-$Laplacian. We study this inequality when $\Omega$ is a quasi open subset of $D$.
	
	We are  interested the existence of a minimizer for the following problem
	\begin{equation*}
		\min\{ \lambda_{1}(\Omega),\Omega\in \mathcal{A}, |\Omega|\leq c\}
	\end{equation*}
	where $c\in (0, |D|)$ , $\mathcal{A}$ is a family of admissible domain defined by
	\begin{equation*}
		\mathcal{A}=\{ \Omega\subseteq D, \Omega\ \mbox{is quasi-open} \}	
	\end{equation*}
	and $\lambda_{1}(\Omega)$ is defined by
	\begin{equation*}
		\lambda_{1}(\Omega):=\inf_{\phi\in W_0^{1,p}(\Omega)\smallsetminus \{0\}}\dfrac{\frac{1}{p}\int_{\Omega}\arrowvert \nabla \phi\arrowvert ^{p}dx+\frac{1}{2}\int_{\Omega}\arrowvert \nabla \phi\arrowvert ^{2}dx}{\frac{1}{q}\int_{\Omega}|\phi|^{q}}=\dfrac{\frac{1}{p}\int_{\Omega}\arrowvert \nabla u_1\arrowvert ^{p}dx+\frac{1}{2}\int_{\Omega}\arrowvert \nabla u_1\arrowvert ^{2}dx}{\frac{1}{q}\int_{\Omega}|u_1|^{q}}.
	\end{equation*}
	The Sobolev space $W^{1,p}_0(\Omega)$ is seen as a closed subspace of $W^{1,p}_0(D)$ defined by
	\begin{equation*}
		W^{1,p}_0(\Omega)=\left\lbrace u\in W^{1,p}_0(D): u= 0 \  p-q.e \ on\ D\backslash\Omega\right\rbrace .
	\end{equation*}
	The problem is to look for weak topology constraints which would make the class $\mathcal{A}$ sequentially compact. This convergence is called weak  $\gamma_p-$convergence for quasi open sets.
	\begin{Def}
		We say that a sequence $(\Omega_n)$ of $\mathcal{A}$ weakly $\gamma_p-$converges to $\Omega$ in $\mathcal{A}$ if the sequence $u_n$ converges weakly in $W^{1,p}_0(D)$ to a function $u\in W^{1,p}_0(D)$ (that we may take as quasicontinuous) such that $\Omega=\left\lbrace u>0\right\rbrace .$
	\end{Def}
	We have the following theorem.
	\begin{thm}\label{exi 1}
		The problem 
		\begin{eqnarray}\label{p2}
			\min\left\lbrace  \lambda_{1}(\Omega),\Omega\in \mathcal{A}, |\Omega|\leq c\right\rbrace 
		\end{eqnarray}
		admits at least one solution.
	\end{thm}
	
	\begin{proof}[Proof of \ref{exi 1} ]
		Let us take
		\begin{equation*}
			\lambda_{1}(\Omega_n):=\inf_{\phi_n\in W_0^{1,p}(\Omega_n)\smallsetminus \{0\}}\dfrac{\frac{1}{p}\int_{\Omega_n}\arrowvert \nabla \phi_n\arrowvert ^{p}dx+\frac{1}{2}\int_{\Omega_n}\arrowvert \nabla \phi_n\arrowvert ^{2}dx}{\frac{1}{q}\int_{\Omega_n}|\phi_n|^{q}dx}=\dfrac{\frac{1}{p}\int_{\Omega_n}\arrowvert \nabla u_n\arrowvert ^{p}dx+\frac{1}{2}\int_{\Omega_n}\arrowvert \nabla u_n\arrowvert ^{2}dx}{\frac{1}{q}\int_{\Omega_n}|u_n|^{q}dx}
		\end{equation*}
		where the minimum is attained by $u_n$, and
		\begin{equation*}
			\lambda_{1}(\Omega):=\inf_{\phi\in W_0^{1,p}(\Omega)\smallsetminus \{0\}}\dfrac{\frac{1}{p}\int_{\Omega}\arrowvert \nabla \phi\arrowvert ^{p}dx+\frac{1}{2}\int_{\Omega}\arrowvert \nabla \phi\arrowvert ^{2}dx}{\frac{1}{q}\int_{\Omega}|\phi|^{q}}=\dfrac{\frac{1}{p}\int_{\Omega}\arrowvert \nabla u_1\arrowvert ^{p}dx+\frac{1}{2}\int_{\Omega}\arrowvert \nabla u_1\arrowvert ^{2}dx}{\frac{1}{q}\int_{\Omega}|u_1|^{q}}
		\end{equation*}
		where the minimum is attained by $u_1.$\\
		Suppose that  $(\Omega_n)_{n\in \mathbb{N}}$ is a minimizing sequence of domain for the problem (\ref{p2}). We denote
		by $u_n$ a first eigenfunction on $\Omega_n$.\\
		Since $u_n$ is the first eigenfunction of $\lambda(\Omega_n), $ $u_n$ is strictly positive (cf \cite{Rocard1}), then the sequence $(\Omega_n)$ is defined by $\Omega_n=\{u_n>0\}.$ \\
		If the sequence $(u_n)$ is bounded in $W^{1,p}_0(D),$ then there exists a subsequence still denoted by $u_n$ such that $u_n$ converges weakly in $W^{1,p}_0(D),$ to a function $u$. \\
		Let $\Omega$ be quasi open and defined by $\Omega=\{u>0\}$, this implies that $u\in W^{1,p}_0(\Omega).$ As the
		sequence $(u_n)$ is bounded in $W^{1,p}_0(D),$ then
		\begin{eqnarray*}
			\lim\limits_{n\rightarrow+ \infty}
			\inf\dfrac{\frac{1}{p}\int_{\Omega_n}\arrowvert \nabla u_n\arrowvert ^{p}dx+\frac{1}{2}\int_{\Omega_n}\arrowvert \nabla u_n\arrowvert ^{2}dx}{\frac{1}{q}\int_{\Omega_n}|u_n|^{q}dx}
			&\geq&\dfrac{\frac{1}{p}\int_{\Omega}\arrowvert \nabla u\arrowvert ^{p}dx+\frac{1}{2}\int_{\Omega}\arrowvert \nabla u\arrowvert ^{2}dx}{\frac{1}{q}\int_{\Omega}|u|^{q}dx}\nonumber\\
			&\geq&\dfrac{\frac{1}{p}\int_{\Omega}\arrowvert \nabla u_1\arrowvert ^{p}dx+\frac{1}{2}\int_{\Omega}\arrowvert \nabla u_1\arrowvert ^{2}dx}{\frac{1}{q}\int_{\Omega}|u_1|^{q}dx}\nonumber\\
			&=& \lambda_{1}(\Omega).
		\end{eqnarray*}
		Now we show that $|\Omega|\leq c.$\\
		We know that if the sequence $\Omega_n$ weakly   $\gamma_p-$ converges to $\Omega$ and the Lebesgue measure is weakly $\gamma_p-$lower semicontinuous on the class $\mathcal{A}$ (see \cite{BUCUR2002}), then we obtain $|\{u > 0\}| \leq \lim\limits_{n\rightarrow +\infty}\inf  |\{u_n > 0\}| \leq c$ this implies that $|\Omega| \leq c$.	
	\end{proof}
	\section{Lipschitz continuity of the first eigenfunction }
	In this section, we focus on the regularity of the optimal shape of the following problem:
	\begin{equation}\label{glm}
		\min\{\lambda_{1}(\Omega),\Omega\subset D ,\ \Omega \  \mbox{quasi-ouvert }, |\Omega|=c \}
	\end{equation}
	where $D$ is an open in $\mathbb{R}^n$ and $c\in (0, |D|).$ 
	It is shown in \cite{Rocard1} , that if $D$ contains a ball of volume $c$, then this ball is a solution of the problem, and is moreover unique, up to translations (and zero-capacity sets).
	\subsection{Free boundary formulation}
	\noindent\\
	We first give an equivalent version of the problem (\ref{glm}) as a free boundary problem, namely an optimization problem in $W^{1,p}_0(D)$ where the domains are sets of level functions.\\
	\textbf{Notation:} For $w\in W^{1,p}_0(D)$, we will denote $\Omega_w=\{x\in D, w(x)\neq0\}.$
	Recall that for a quasi-open bounded subset $\Omega$ of $D$
	
	\begin{equation}\label{min}
		\lambda_{1}(\Omega)=\min \left\lbrace  \dfrac{\frac{1}{p} \displaystyle\int_{\Omega}\arrowvert \nabla u\arrowvert ^{p}dx+\frac{1}{2}\displaystyle\int_{\Omega}\arrowvert \nabla u\arrowvert ^{2}dx}{\frac{1}{q}\displaystyle\int_{\Omega}|u|^{q}dx}, \ \ u \in W_0^{1,p}(\Omega)\smallsetminus \{0\} \right\rbrace 
	\end{equation}
	\begin{Def}\label{def11}
		In this subsection, we denote by $u_{\Omega}$ any positive minimizer in (\ref{min}), that is such that
		\begin{equation}
			u_{\Omega} \in W^{1,p}_{0}(\Omega)\smallsetminus \{0\},\ \dfrac{\frac{1}{p} \displaystyle\int_{\Omega}\arrowvert \nabla u_{\Omega}\arrowvert ^{p}dx+\frac{1}{2} \displaystyle\int_{\Omega}\arrowvert \nabla u_{\Omega}\arrowvert ^{2}dx}{\frac{1}{q}\displaystyle\int_{\Omega}|u_{\Omega}|^{q}dx}=\lambda_{1}(\Omega).
		\end{equation}
	\end{Def}
	\begin{rem}\label{rem3.3}
		Select from (\ref{min}), $\nu=\nu(t):u_{\Omega}+ t\phi$ avec $\phi\in W^{1,p}_0(\Omega)$ and using that the derivative at $t=0$ for $t\mapsto\dfrac{\frac{1}{p}\displaystyle\int_{\Omega}\arrowvert \nabla \nu(t)\arrowvert ^{p}dx+\frac{1}{2}\displaystyle\int_{\Omega}\arrowvert \nabla \nu(t)\arrowvert ^{2}dx}{\frac{1}{q}\displaystyle\int_{\Omega}|\nu(t)|^{q}dx}$  vanishes leads to
		\begin{equation}\label{vari12c}
			\forall \phi \in W^{1,p}_0(\Omega),\ \ 	\displaystyle\int_{\Omega}| \nabla u_{\Omega}|^{p-2}\nabla u_{\Omega} \nabla\phi dx +\displaystyle\int_{\Omega}\nabla u_{\Omega} \nabla\phi dx=\lambda_1(\Omega)\displaystyle\int_{\Omega} |u_{\Omega}|^{q-2}u_{\Omega}\phi dx.
		\end{equation} 
		If $\Omega$ is an open set, (\ref{vari12c}) means exactly that $	-\Delta_p u -\Delta 
		u   = \lambda_1 |u|^{q-2}u$  in the sense of the distributions in $\Omega$.
		
	\end{rem}
	Note that if $u_{\Omega}$ is a minimizer in (\ref{min}), so is $|u_{\Omega}|$. Therefore, without lost of generality, we can assume that $u_{\Omega}\geq0$ and we will always make this assumption in this section on the minimization of $\lambda_{1}(\Omega)$. If $\Omega$ is a connected open set then $u_{\Omega}>0$ on $\Omega.$ This is a consequence of the maximum principle applied to $-\Delta_p u -\Delta 
	u   = \lambda_1(\Omega) |u|^{q-2}u\geq 0.$ This extends (quasi-everywhere) to the case when $\Omega$ is a quasi-connected quasi-open set, but the proof requires a little more computation.
	Since $\Omega \mapsto \lambda_{1}(\Omega)$ is nonincreasing with respect to inclusion, any solution of(\ref{glm}) is also solution of 
	\begin{eqnarray}\label{p22}
		\min\{ \lambda_{1}(\Omega),\Omega\in \mathcal{A}, |\Omega|\leq c\}
	\end{eqnarray}
	The converse is true in most situations, in particular if $D$ is connected. Note that it may happen that if $D$ is not connected, then a solution to (\ref{p22}) does not satisfy $|\Omega|=c.$\\
	We will first consider the problem (\ref{p22}), and this will nevertheless provide a complete
	understanding of (\ref{glm}). We start by proving that (\ref{p22}) is equivalent to a free boundary problem.
	\begin{prop}\label{prop3.4}
		\begin{enumerate}
			\item Let $\Omega^*$ be a quasi-open solution of the minimization problem (\ref{p22}) and let us set $u=u_{\Omega^*}$. Then 
			\begin{eqnarray}\label{p222}
				\dfrac{\frac{1}{p}\displaystyle\int_{D}\arrowvert \nabla u\arrowvert ^{p}dx+\frac{1}{2}\displaystyle\int_{D}\arrowvert \nabla u\arrowvert ^{2}dx}{\frac{1}{q}\displaystyle\int_{D}|u|^{q}dx}=\min_{ v\in  E} \left\lbrace  \dfrac{\frac{1}{p}\displaystyle\int_{D}\arrowvert \nabla v\arrowvert ^{p}dx+\frac{1}{2}\displaystyle\int_{D}\arrowvert \nabla v\arrowvert ^{2}dx}{\frac{1}{q}\displaystyle\int_{D}|v|^{q}dx} \right\rbrace,
			\end{eqnarray}
		with $E=\{  v \in W_0^{1,p}(D)\smallsetminus \{0\}, \ |\Omega_v|\leq c\}.$
			\item Let $u$ be the solution of the minimization problem (\ref{p222}), then $\Omega_u$ is a solution of (\ref{p22}).
		\end{enumerate}	
	\end{prop}
	\begin{proof}[Proof of the proposition \ref{prop3.4}]
		For the first point, we choose $v\in W^{1,p}_{0}(D)$, with $|\Omega_v|\leq c$ and we apply (\ref{p22}). This gives $\dfrac{\frac{1}{p}\displaystyle\int_{D}\arrowvert \nabla u\arrowvert ^{p}dx+\frac{1}{2}\displaystyle\int_{D}\arrowvert \nabla u\arrowvert ^{2}dx}{\frac{1}{q}\displaystyle\int_{D}|u|^{q}dx}=\lambda_{1}(\Omega^*)\leq \lambda_{1}(\Omega_v)$ and we have
		
		\begin{eqnarray*}\label{p223}
			\dfrac{\frac{1}{p}\displaystyle\int_{D}\arrowvert \nabla u\arrowvert ^{p}dx+\frac{1}{2}\displaystyle\int_{D}\arrowvert \nabla u\arrowvert ^{2}dx}{\frac{1}{q}\displaystyle\int_{D}|u|^{q}dx}\leq\min \left\lbrace  \dfrac{\frac{1}{p}\displaystyle\int_{D}\arrowvert \nabla v\arrowvert ^{p}dx+\frac{1}{2}\displaystyle\int_{D}\arrowvert \nabla v\arrowvert ^{2}dx}{\frac{1}{q}\displaystyle\int_{D}|v|^{q}dx}, \ \ v \in W_0^{1,p}(D)\smallsetminus \{0\}, \ |\Omega_v|\leq c \right\rbrace
			.
		\end{eqnarray*}
		
		Equality holds since $u\in  W^{1,p}_{0}(\Omega^*)\subset  W^{1,p}_{0}(D)$ with $\Omega^*=\{ u>0\}=\Omega_u\subset D$ and $|\Omega_u|=|\Omega^*|\leq c.$\\
		For the second point, let $u$ be a solution of (\ref{p222}). Then $|\Omega_u|\leq c.$  Let $\Omega\subset D$ quasi-open with $|\Omega|\leq c$ and let $u_{\Omega}$ as in the Definition \ref{def11}.  Then
		
		\begin{eqnarray*}
			\lambda_{1}(\Omega_u)&\leq& \dfrac{\frac{1}{p}\displaystyle\int_{D}\arrowvert \nabla u\arrowvert ^{p}dx+\frac{1}{2}\displaystyle\int_{D}\arrowvert \nabla u\arrowvert ^{2}dx}{\frac{1}{q}\displaystyle\int_{D}|u|^{q}dx}\\
			& \leq& \dfrac{\frac{1}{p}\displaystyle\int_{D}\arrowvert \nabla u_{\Omega}\arrowvert ^{p}dx+\frac{1}{2}\displaystyle\int_{D}\arrowvert \nabla u_{\Omega}\arrowvert ^{2}dx}{\frac{1}{q}\displaystyle\int_{D}|u_{\Omega}|^{q}dx}\\
			&=&\lambda_{1}(\Omega).
		\end{eqnarray*}
	\end{proof}
	\begin{rem}\label{remarque 5}
		We will now work with the functional problem (\ref{p222}) rather than (\ref{p22}). Note that if $D$ is bounded (or with finite measure), the existence of the minimum $u$ follows easily from the compactness of $W^{1,p}_{0}(D)$ in $L^p(D)$ applied to a minimizing sequence (that we may assume to be weakly convergent in $W^{1,p}_{0}(D)$ and strongly in $L^p(D)$).
	\end{rem}
	
	\begin{rem}\label{rem3.6}
		Two different situations may occur. If $D$ is connected and $\Omega^*$ is a solution of (\ref{p22}), then $c^*:= |[u_{\Omega^*}>0]|=c$ and $\Omega^*=[u_{\Omega^*}>0].$ If $D$ is not connected, it may happen that $ c^* < c$ and therefore $u_{\Omega^*} > 0$ on some of the connected components of $D$ and identically zero on the others.\\	
		Indeed, if $ c^* < c,$ then for all balls $B \subset D$ whith measure is less than $c-c^*$ and for all $\varphi \in W^{1,p}_{0}(B) $, we can choose, $\nu=\nu(t): =u+ t\varphi$ with $u:= u_{\Omega^*}\geq0$ in (\ref{p222}). Using that the derivative at  $t=0$ for $t\mapsto\dfrac{\frac{1}{p}\displaystyle\int_{D}\arrowvert \nabla \nu(t)\arrowvert ^{p}dx+\frac{1}{2}\displaystyle\int_{D}\arrowvert \nabla \nu(t)\arrowvert ^{2}dx}{\frac{1}{q}\displaystyle\int_{D}|\nu(t)|^{q}dx}$  vanishes leads to
		\begin{equation}\label{vari12c}
			\forall \phi \in W^{1,p}_0(D),\ \ 	\int_{\Omega}| \nabla u_{\Omega}|^{p-2}\nabla u_{\Omega} \nabla\phi dx +\int_{\Omega}\nabla u_{\Omega} \nabla\phi dx=\lambda_1(\Omega)\int_{\Omega} |u_{\Omega}|^{q-2}u_{\Omega}\phi dx,
		\end{equation}
		with $  \lambda_c:=\dfrac{\frac{1}{p}\displaystyle\int_{D}\arrowvert \nabla u\arrowvert ^{p}dx+\frac{1}{2}\displaystyle\int_{D}\arrowvert \nabla u\arrowvert ^{2}dx}{\frac{1}{q}\displaystyle\int_{D}|u|^{q}dx}$. This implies that $	-\Delta_p u -\Delta u   = \lambda_c |u|^{q-2}u$  in the sense of the distributions in $D$. The strict maximum principle implies that in each connected component of $D$, either $u>0$, or $u\equiv 0$. If $D$ is connected, we have $\{u>0\}=D$ and we get a contradiction since $c <|D|$  and $c^*=|\Omega^*|=|D|$. Therefore necessarily $c^* = c$ if $D$ is connected.	\\ \\
	\end{rem}
	\subsection{Existence and Lipschitz regularity of the state function }
	\noindent\\
	\subsubsection{Equivalence with a penalized version}
	\noindent\\
	We shall first prove that (\ref{p222}) is equivalent to a penalized version
	\begin{prop}\label{propvp}
		Suppose that $|D|<+\infty.$ Let u be a solution of (\ref{p222}) and 
		\begin{equation*}
			\lambda_c:=\dfrac{\frac{1}{p}\displaystyle\int_{D}\arrowvert \nabla u\arrowvert ^{p}dx+\frac{1}{2}\displaystyle\int_{D}\arrowvert \nabla u\arrowvert ^{2}dx}{\frac{1}{q}\displaystyle\int_{D}|u|^{q}dx}
		\end{equation*}
		Then, there exists $\mu > 0$ such that
		\begin{eqnarray}\label{p225}
			\dfrac{\frac{1}{p}\displaystyle\int_{D}\arrowvert \nabla u\arrowvert ^{p}dx+\frac{1}{2}\displaystyle\int_{D}\arrowvert \nabla u\arrowvert ^{2}dx}{\frac{1}{q}\displaystyle\int_{D}|u|^{q}dx}&\leq&  \frac{1}{p}\int_{D}|\nabla v|^{p}+ \frac{1}{2}\displaystyle\int_{D}| \nabla v|^{2} + \lambda_c\left[ 1-\frac{1}{q}\displaystyle\int_{D}| v|^{q}\right]\\
			& & +\mu\left[|\Omega_v|- c\right]^+,  \forall v\in  W^{1,p}_{0}(D).\nonumber
		\end{eqnarray}
	\end{prop}
	\begin{proof}[Proof of the proposition \ref{propvp} ]
		Note first that, by definition of $u$ and $\lambda_{c}$, for any $ v\ in \ W^{1,p}_{0}(D)$ with $|\Omega_v|\leq c$, we have $\frac{1}{p}\displaystyle\int_{D}|\nabla v|^{p}+ \frac{1}{2}\displaystyle\int_{D}| \nabla v|^{2} - \frac{\lambda_c}{q}\displaystyle\int_{D}|v|^{q}\geq 0$ or
		\begin{eqnarray}\label{p226}
			\dfrac{\frac{1}{p}\displaystyle\int_{D}\arrowvert \nabla u\arrowvert ^{p}dx+\frac{1}{2}\displaystyle\int_{D}\arrowvert \nabla u\arrowvert ^{2}dx}{\frac{1}{q}\displaystyle\int_{D}|u|^{q}dx}\leq  \frac{1}{p}\displaystyle\int_{D}|\nabla v|^{p}+ \frac{1}{2}\displaystyle\int_{D}| \nabla v|^{2} + \lambda_c\left[ 1-\frac{1}{q}\displaystyle\int_{D}| v|^{q}\right]\ \forall v\in  W^{1,p}_{0}(D).\nonumber\\
		\end{eqnarray}
		Let us now denote by $J_\mu(v)$ the right-hand side of (\ref{p225}) and let $u_\mu$ be a minimizer of $J_\mu(v)$ for $v\in W^{1,p}_{0}(D)$ (its existence follows from the compactness of $v\in W^{1,p}_{0}(D)$ in $L^p(D)$, see also Remark \ref{remarque 5} ). If we replace $u_\mu$ by $|u_\mu|$, we can suppose that $u_\mu \geq0$.\\
		For the conclusion of the proposition, it is sufficient to prove $|\Omega_{u_\mu}| \leq c$ since then 
		\begin{eqnarray*}
			J_{\mu}(u_\mu)\leq J_\mu(u)=\lambda_c=\dfrac{\frac{1}{p}\displaystyle\int_{D}\arrowvert \nabla u\arrowvert ^{p}dx+\frac{1}{2}\displaystyle\int_{D}\arrowvert \nabla u\arrowvert ^{2}dx}{\frac{1}{q}\displaystyle\int_{D}|u|^{q}dx}\leq	J_{\mu}(u_\mu)
		\end{eqnarray*}
		where the latter inequality comes from (\ref{p226}).\\
		In order to obtain a contradiction, let us suppose that $|\Omega_{u_\mu}|>c$ and introduce $u^t :=(u_\mu - t)^+$ where $t>0$ small enough. Then $J_{\mu}(u_\mu)\leq J_{\mu}(u^t)$ we have:
		\begin{eqnarray*}
			& &\frac{1}{p}\displaystyle\int_{D}|\nabla u_\mu |^{p}+ \frac{1}{2}\displaystyle\int_{D}| \nabla u_\mu|^{2} + \lambda_c\left[ 1-\frac{1}{q}\int_{D}|u_{\mu}|^q\right] +\mu\left[\Omega_{u_\mu}|- c\right]^+\nonumber\\
			&\leq& \frac{1}{p}\displaystyle\int_{D}|\nabla u^t |^{p}+ \frac{1}{2}\displaystyle\int_{D}| \nabla u^t|^{2} + \lambda_c\left[ 1-\frac{1}{q}\displaystyle\int_{D}|u^t|^q\right] +\mu\left[|\Omega_{u^t}|- c\right]^+\nonumber\\
			&=&\frac{1}{p}\displaystyle\int_{[u_\mu\geq t]}|\nabla u_\mu |^{p}+ \frac{1}{2}\displaystyle\int_{[u_\mu\geq t]}| \nabla u_\mu|^{2} + \lambda_c\left[ 1-\frac{1}{q}\int_{[u_\mu\geq t]}|u_\mu - t|^q \right] +\mu\left[|\Omega_{u^t}|- c\right]^+\nonumber\\
		\end{eqnarray*}
	Using $\Omega_{u_\mu}=\{0<u_\mu<t\}\cup\{u_\mu\geq t\}$ and $|\Omega_{u^t}| >c $ for $t$ small enough, we have
	\begin{eqnarray*}
		\frac{1}{p}\displaystyle\int_{[0<u_\mu<t]}|\nabla u_\mu|^{p}+ \frac{1}{2}\displaystyle\int_{[0<u_\mu<t]}| \nabla u_\mu|^{2}  +\mu|\left[0<u_\mu<t\right]|&\leq& \frac{\lambda_c}{q}\displaystyle \int_{[0<u_\mu<t]}|u_\mu|^q\\
		 & &+\frac{\lambda_c}{q}\displaystyle\int_{[u_\mu\geq t]} |u_\mu|^q-|u_\mu - t|^q.
	\end{eqnarray*}
Consider the function $x\in (0,\infty)\mapsto B(x)= x^q-1-q(x-1), \ q\geq 2.$ We have $B(1)=0$ and $B'(x)=q(x^{q-1}-1)$ is zero for $x = 1$ only, at this point, $B(x)$ takes its only minimum value, which is zero. Therefore $B(x)\geq0$ for all $x.$ Taking $x=1-\frac{t}{u_\mu}$ for $u_\mu> t$ we have $x>0$ and $B(x)=(1-\frac{t}{u_\mu})^q-1-q(1-\frac{t}{u_\mu}-1)\geq0$ i.e
$(1-\frac{t}{u_\mu})^q\geq 1-\frac{tq}{u_\mu}$ thus $|u_\mu - t|^q\geq |u_\mu|^q - qt|u|^{q-1}.$\\
We therefore obtain
	\begin{eqnarray*}
	\frac{1}{p}\displaystyle\int_{[0<u_\mu<t]}|\nabla u_\mu|^{p}+ \frac{1}{2}\displaystyle\int_{[0<u_\mu<t]}| \nabla u_\mu|^{2}  +\mu|\left[0<u_\mu<t\right]|
	&\leq& \frac{\lambda_c}{q}\displaystyle \int_{[0<u_\mu<t]}|u_\mu|^q+\lambda_c t\displaystyle\int_{[u_\mu\geq t]} |u_\mu|^{q-1}\\
	&\leq&\lambda_c t |[0<u_\mu<t]|+\lambda_c t\|u_\mu\|_{L^p(D)}^{q-1}|[u_\mu\geq t]|^{\frac{p+1-q}{p}}\\
	&\leq& tK, 
\end{eqnarray*}
with $K=K(p,q,\lambda_c,\|u_\mu\|_{L^p}|,|\Omega_{u_\mu}|).$
		And moreover, we have
		\begin{eqnarray*}
			\frac{1}{p}\displaystyle\int_{[0<u_\mu<t]}|\nabla u_\mu|^{p} +\mu|\left[0<u_\mu<t\right]|\leq 	\frac{1}{p}\displaystyle\int_{[0<u_\mu<t]}|\nabla u_\mu|^{p}+ \frac{1}{2}\displaystyle\int_{[0<u_\mu<t]}| \nabla u_\mu|^{2}  +\mu|\left[0<u_\mu<t\right]|
		\end{eqnarray*}
		Using the coaréa formula (see for example \cite{Evans}, \cite{Giusti}), this can be rewritten as
		\begin{eqnarray*}
			\displaystyle\int_{0}^{t}ds\displaystyle\int_{[u_\mu=s]}\left[ \frac{1}{p}|\nabla u_\mu|^{p-1}+\frac{\mu}{| \nabla u_\mu|}\right] d\mathcal{H}^{d-1}\leq t K
		\end{eqnarray*}
		But the function $x\in (0,\infty)\mapsto \frac{1}{p}x^{p-1}+\mu x^{-1}\in [0,\infty)$ is bounded by the bottom $\left( \frac{p}{p-1}\right) ^{1-\frac{1}{p}} \mu ^{1-\frac{1}{p}}.$ Therefore, it follows that
		\begin{eqnarray}
		\left( \frac{p}{p-1}\right) ^{1-\frac{1}{p}} \mu ^{1-\frac{1}{p}}\displaystyle\int_{0}^{t}ds\int_{[u_\mu=s]} d\mathcal{H}^{d-1}\leq t K
		\end{eqnarray}
		We now use the isoperimetric inequality $\displaystyle\int_{[u_\mu=s]} d\mathcal{H}^{d-1}\geq C(d)|[u_\mu >s]|^{\frac{d-1}{d}}$. We divide the inequality by $t$ and let $t\rightarrow 0$, to deduce
		
		\begin{eqnarray*}
			\left( \frac{p}{p-1}\right) ^{1-\frac{1}{p}} \mu ^{1-\frac{1}{p}}C(d)|\Omega_{u_\mu}|^{\frac{d-1}{d}}\leq  K
		\end{eqnarray*}
		Thus, $|\Omega_{u\mu}|>c$ is impossible if $\mu>\mu^*:= K^{\frac{p}{p-1}}\left( \frac{p-1}{p}\right) C(d)^{-\frac{p}{p-1}}c^{\frac{(1-d)p}{d(p-1)}}$. Therefore, the conclusion of the proposition \ref{propvp} is valid for any $\mu >\mu^*.$
	\end{proof}
	\begin{rem}
	Given a quasi-open set $\Omega\subset D$ and choosing $v=u_{\Omega}$ in (\ref{p225}), we obtain the penalized "domain" version of (\ref{p22}), where $\Omega^*$ is a solution of (\ref{p22})
	\begin{eqnarray}
		\lambda_{1}(\Omega^*)\leq \lambda_{1}(\Omega)+\frac{\mu q}{\frac{1}{q}\displaystyle\int_{D}| u_{\Omega}|^{q}dx}\left[|\Omega_{u_{\Omega}}|- c\right]^+, \  \forall \Omega\subset D, \ \Omega \ \mbox{quai-open}.
	\end{eqnarray}
\end{rem}
	\begin{rem}[Sub- and super-solutions] 
		Note that to prove the Proposition \ref{propvp}, we use only perturbations of the optimal domain $\Omega_u$ from inside. This means that the same
		result is valid for shape subsolutions where  (\ref{p222}) is assumed only for functions $v$ for which $\Omega_v\subset\Omega_u.$
	\end{rem}
	Next, we shall prove Lipschitz continuity of the functions $u$ solutions of the penalized problem (\ref{p225}). Interestingly, Lipschitz continuity will hold for super-solutions of (\ref{p225}) which are defined when the inequality (\ref{p225}) is valid only for perturbations from outside, i.e. such that $\Omega_u\subset\Omega_v.$
	\subsubsection{A general sufficient condition for Lipschitz regularity}
	\noindent\\
	We use the notation $\fint_{\partial B(x_0,r) }U(x)d\sigma(x)$ to denote the average of $U$ over $\partial B(x_0,r_0)$
	\begin{Lem}\label{Lem5.1}
		
		Let $B(x_0,r_0)\subset D$ and $U\in C^{2}(B(x_0,r_0))$. Then, for all $r\in (0,r_0)$
		\begin{eqnarray}\label{for3.17}
			\fint_{\partial B(x_0,r)}U(x)d\sigma(x)- U(x_0)&=& c(d)\int_{\rho}^r s^{1-d}\left[ \int_{B(x_0,r)} \Delta U\right] ds\\
			&=& c(d)\int_{\rho}^r s^{1-d}\Delta U(B_s)\\\nonumber	
		\end{eqnarray}	
		This remains valid for all  $U \in W^{1,p}(B(x_0,r))\subset H^1(B(x_0,r))$ such that $\Delta U$ is a measure, satisfying
		
		\begin{equation}\label{equ23}
			\int_{0}^r s^{1 -d}|\Delta U|(B_s)<\infty,
		\end{equation}
		and $U$ is then pointwise defined by
		\begin{eqnarray}\label{fol1}
			U(x_0)=\lim\limits_{r\rightarrow 0}\fint_{ B(x_0,r)}U(x)dx
		\end{eqnarray}
	\end{Lem}
	\begin{proof}[Proof of Lemma \ref{Lem5.1}]
		See \cite{Briancon}
		
		%		
		%		 We assume that $x_0 = 0$ and we calculate
		%		\begin{eqnarray}
		%			\frac{d}{ds}\fint_{\partial B(s)}Ud\mathcal{H}^{d-1}&=&\frac{d}{ds}\fint_{\partial B(1)}U(s\xi)d\mathcal{H}^{d-1}\nonumber\\
		%			&=&\fint_{\partial B(1)}\xi .\nabla U(s\xi)d\mathcal{H}^{d-1}\nonumber\\
		%			&=&\frac{s^{1-d}}{d\omega_d}div(\nabla u)(B(s))
		%		\end{eqnarray}
		%		Then, integrating from $\rho$ to $r$ $(\rho < r)$, we have
		%		\begin{eqnarray}
		%			\fint_{\partial B_r}Ud\mathcal{H}^{d-1}-\fint_{\partial B_{\rho}}Ud\mathcal{H}^{d-1}=c(d)\int_{\rho}^r s^{1-d}div(\nabla u)(B_s)
		%		\end{eqnarray}
		%		It extends to functions $U\in W^{1,p}(D)$ where $div(\nabla u)$ is a measure with 
		%		\begin{equation}\label{equ23}
		%			\int_{0}^r s^{1 -d}|div(\nabla U)|(B_s)<\infty,
		%		\end{equation}
		%	
		%	
		%		
		%		. We can then consider that $U$ is precisely defined in $x_0$.
		%		\begin{eqnarray*}
		%			U(x_0)=\lim\limits_{r\rightarrow 0}\fint_{\partial B_r(x_0,r)}U(x)d\mathcal{H}^{d-1}(x)=\lim\limits_{r\rightarrow 0}\fint_{ B(x_0,r)}U(x)dx.
		%		\end{eqnarray*}
		%		Thus (2) is established.
	\end{proof}
\begin{Lem}\label{Lema4.22}
	Let $B(x_0,r_0)\subset D, \ r_0\leq 1, F \in  L^q(B(x_0,r_0)),\ q>d.$ Then, there exists $C=C(\|F\|_{B(x_0,r_0)},d)$ such that, for $r\in (0,r_0),$
	\item[-] if $\Delta U= F$ on $B(x_0,r_0)$, then 
\begin{equation*}
	\| \nabla U\|_{L^{\infty}(B(x_0,r/2))}\leq C [1+r^{-1}\|U\|_{L^{\infty}(B(x_0,r_0))}],
\end{equation*}
\item[-]if $\Delta U\geq F$ and $U\geq0$ on $B(x_0,r_0),$ then 
\begin{equation*}
	\| U\|_{L^{\infty}(B(x_0,2r/3))}\leq C [r+ \fint_{\partial B(x_0,r)}U].
\end{equation*}
\end{Lem}
\begin{proof}[Proof of Lemma \ref{Lema4.22} ]
See \cite{Briancon}	
\end{proof}	
\begin{Lem}\label{remlem}
	Let $U \in W^{1,p}_0(B(x_0,r))\subset H^1_0(B(x_0,r))$ we have
	\begin{equation*}\
		\int_{0}^r s^{1 -d}|\Delta U|(B_s)\leq \int_{0}^r s^{1 -d}|\Delta_p U+\Delta U|(B_s).
	\end{equation*}
	Moreover, it is natural to see that if $\Delta_p U+\Delta U$ is a measure then $\Delta U$ is also.	
\end{Lem}
\begin{proof}[Proof of Lemma \ref{remlem}]
Let $U \in W^{1,p}_0(B(x_0,r))\subset H^1_0(B(x_0,r))$. Let $\varphi \in C^{\infty}_0(B(x_0,r))$
\begin{equation*}
\langle \Delta_p U+\Delta U ,\varphi\rangle =\int_{B(x_0,r)}\left( |\nabla U|^{p-2}+1\right) \nabla U \nabla\varphi,
\end{equation*}
we see that the sign of  $\left( |\nabla U|^{p-2}+1\right) \nabla U \nabla\varphi$ is that of $\nabla U \nabla\varphi. $ Thus we have $|\Delta_p U+\Delta U|(B_s)\geq |\Delta U|(B_s)$ and if $\Delta_p U+\Delta U\geq0$ then $\Delta U\geq0.$
\end{proof}
	\begin{prop}\label{prop11}
		Let $U\in W^{1,p}_{0}(D) $ bounded and continuous on D and let $\omega:=\{x\in D, U(x)\neq 0 \}.$ Suppose that $\Delta_p U+\Delta U$ is a measure such that $\Delta_p U+\Delta U=g$ on $\omega$ with $g\in L^{\infty}(\omega)$ and
		\begin{eqnarray}\label{lar1}
			|\Delta_p |U|+\Delta |U||(B(_0,r))\leq Cr^{d-1}
		\end{eqnarray}
		for all $x_0\in D$ with $B(x_0,r)\subset D, r\leq 1$ and $U(x_0)=0.$ Then $U$ is locally Lipschitz continuous on $D$. If furthermore $D = \mathbb{R}^n$, then $U$ is globally Lipschitz continuous.
	\end{prop}
	\begin{rem}
		Note that if $U$ is locally Lipschitz continuous on $D$ with $\Delta_p U+\Delta U\geq 0$, then for a test function $\varphi$ with
		\begin{eqnarray}\label{eq3.16}
			\varphi \in C^{\infty}_0(B(x_0,2r)),B(x_0,2r)\subset D, 0\leq \varphi\leq 1\nonumber\\
			\varphi\equiv 1 \ sur \ B(x_0,r), \|\nabla\varphi\|_{L^{\infty}(B)}\leq C/r.
		\end{eqnarray}
		we have
		\begin{eqnarray*}
			(	\Delta_p U+\Delta U)B(x_0,r)\leq \int_{D}\varphi d(\Delta_p U+\Delta U)=\int_{D}|\nabla U|^{p-2}\nabla U\nabla \varphi+\int_{D}\nabla U\nabla \varphi\nonumber\\
			\leq\|\nabla U\|_{L^{\infty}}^{p-1}+\|\nabla U\|_{L^{\infty}}|\Omega_{\varphi}|\|\nabla \varphi\|\leq C(\|\nabla U\|_{L^{\infty}}^{p-1}+\|\nabla U\|_{L^{\infty}})r^{d-1}.
		\end{eqnarray*}
	\end{rem}
	This indicates that the estimate (\ref{lar1}) is essentially a necessary condition for the Lipschitz continuity of $U$. This theorem states that the converse holds in some cases which are relevant for our analysis as it will appear in the next sub-section.
	%\begin{rem}
	%	Dans la démonstration de la proposition \ref{prop11}, comme dans \cite{Briancon}, nous utiliserons l'identité suivante qui est utile pour estimer la variation des fonctions : 
	%	
	%\begin{eqnarray}
	%\fint_{\partial B(x_0,r)}U(x)d\sigma(x)- U(x_0)=c(d)\int_0^r s^{1-d}\left[ \int_{B(x_0,r)}d(\Delta_p U+\Delta U)\right] ds.
	%\end{eqnarray}		
	%\end{rem}
	\begin{proof}[Proof of Proposition \ref{prop11}]
		We want to prove that $\nabla U \in L^{\infty}_{loc}(D)$. We can first assert that $\nabla U=0  \ a.e\ \mbox{on } \ D\backslash \omega.$ On the open set $\omega$, we have $\Delta_p U+\Delta U=g\in L^{\infty}(\omega)$ such that at least $U\in C(\omega)$.\\
		Let us note $D_{\delta}=\{x \in D, d(x,\partial D)>\delta \}$ (we start with the case $D\neq \mathbb{R}^d$). We will bound $\nabla U (x_0) \ \mbox{for}\ x_0\in \omega\cap D_{\delta}$. The meaning of the constant $C$ will vary but will always depend only on $\delta$, $\|U\|_{L^{\infty}(D)}, \|g\|_{\infty}(D),$ $d$ and on the constant $C$ in the assumption (\ref{lar1}) .\\
		Let $y_0\in \partial\omega$ such that $|x_0-y_0|=d(x_0,\partial\omega):=r_0$. Then $r_0>0$ and $B(x_0,r_0)\subset\omega.$ We have $U(y_0) = 0$ since $y_0 \in \partial\omega$ and $U$ is continuous. Let us therefore introduce $s_0:=\min\{r_0,1\}, B_0=B(x_0,s_0)$ and $V\in W^{1,p}_0(B_0)$ such that $\Delta V=g-\Delta_p U$ on $B_0.$ Since $ h=g-\Delta_p U\in L^{\infty}(\omega)$ , by scaling we obtain
		\begin{eqnarray*}
			\|V\|_{L^{\infty}(B_0)}\leq C s_0^2; \ \|\nabla
			V\|_{L^{\infty}(B_0)}\leq C s_0; \ C=C(\|h\|_{L^{\infty}}).
		\end{eqnarray*}
		Since $U - V$ is harmonic on $B_0$ we also have $|\nabla (U-V)(x_0)|\leq \frac{d}{s}\|U-V\|_{L^{\infty}(B_0)}$ such that
		\begin{eqnarray}\label{ed54}
			|\nabla U (x_0)|\leq |\nabla V(x_0)|+ds^{-1}\|U-V\|_{L^{\infty}(B_0)}\leq C\left[ s_0 +s_0^{-1} \|U\|_{L^{\infty}(B_0)}\right] .
		\end{eqnarray}
		If $s_0\geq\frac{\delta}{16}$ , we deduce from (\ref{ed54}) $|\nabla U (x_0)|\leq C(\delta,\|U\|_{L^{\infty}(B_0)},\|h\|_{L^{\infty}}).$ We now assume that $\delta\leq16.$\\
		If $s_0< \frac{\delta}{16}$, that is $r_0=s_0<\frac{\delta}{16},$ since $x_0\in D_{\delta}, d(y_0,\partial D)\geq d(x_0,\partial D)-d(x_0,y_0)\geq \delta-r_0\geq 15r_0,$ which implies $B(x_0,r_0)\subset B(y_0,2r_0)\subset B(y_0,8r_0)\subset D.$ Thanks to the hypothesis (\ref{lar1}), $U(y_0)=0$ and to the formula (\ref{for3.17}) applied with $U$ replaced by $|U|$, we deduce $\fint_{\partial B(y_0,4r_0)}|U(x)|d\sigma(z)\leq Cr_0.$ Finally, by using the representation $(U-V)(x)=\fint_{\partial B(y_0,4r_0)}U(x)P_x(z)d\sigma(z)$ for all $x \in B(y_0,2r_0) $ where $P_x(.)$ is the Poisson kernel in $x$ (The Poisson kernel is the way to extend an integral function on the circle into a harmonic function in the disk.), we have
		\begin{eqnarray}
			\|U-V\|_{L^{\infty}(B_0)}\leq\|U-V\|_{L^{\infty}(B(y_0,2r_0))}\leq\fint_{\partial B(y_0,4r_0)}|U(x)|d\sigma(z)\leq Cr_0.
		\end{eqnarray}
		this together with (\ref{ed54}) (where $s_0 = r_0$) and $\|V\|_{L^{\infty}(B_0)}\leq Cr_0^2$, this implies $|\nabla u(x_0)|\leq C.$\\
		Now if $D=\mathbb{R}^d$, then $\omega =\mathbb{R}^d$ and (\ref{ed54}) gives the estimate $(r_0 =+\infty, s_0 = 1)$, or $\omega\neq\mathbb{R}^d$: then we argue just as above, replacing $\frac{\delta}{16}$ by 1 in the discussion.
	\end{proof}
	In the proposition \ref{prop11}, we have assumed that the function $U$ is continuous on $D$. For our optimal eigenfunctions, this will be a consequence of the following lemma.
	\begin{Lem}\label{lem3.14}
		Let $U\in W^{1,p}_{0}(D) $ be such that $\Delta_p U+\Delta U$ is a measure satisfying 
		\begin{eqnarray}\label{lar15}
			|\Delta_p |U|+\Delta |U||(B(x_0,r))\leq Cr^{d-1}
		\end{eqnarray}
		for all $x_0\in D$ with $B(x_0,r)\subset D, r\leq 1.$ Then $U$ is continuous on $D$.
	\end{Lem}
	\begin{proof}[Proof of Lemma\ref{lem3.14} ]
		Let $U\in W^{1,p}_{0}(D)\subset  W^{1,2}_{0}(D)$	The hypothesis (\ref{lar15}) implies that\\ $\displaystyle\int_0^r s^{1-d}\left[ |\Delta_p |U|+\Delta |U|| (B_s\right] ds<\infty$ so $\displaystyle\int_0^r s^{1-d}\left[ \Delta |U|(B(x_0,s))\right] ds<\infty$ so that (\ref{fol1}) and (\ref{for3.17}) hold. Let $x_0, y_0 \in D$ and $r > 0$ be small enough that $B(x_0,2r)\subset D$, $B(y_0,2r)\subset D.$
		We deduce, using again (\ref{lar15}) and the representation (\ref{fol1}) :
		\begin{eqnarray*}
			|U(x_0)-U(y_0)|&\leq&\left|\fint_{\partial B(x_0,r_0)}U- \fint_{\partial B(y_0,r_0)}U\right| +Cr\\
			&\leq&
			|\fint_{\partial B(0,r_0)}|U(x_0+\xi)-U(y_0+\xi)|d\sigma(\xi)+Cr
		\end{eqnarray*}
		But by continuity of the trace operator of $W^{1,p}(B(0,r))$ in $L^1(B(0,r))$, this implies
		
		\begin{eqnarray*}
			|U(x_0)-V(y_0)|\leq	\|U(x_0+.)-U(y_0+.)\|_{W^{1,p}(B(0,r))}+Cr.
		\end{eqnarray*}
		Then
		\begin{eqnarray}
			\lim\limits_{y_0\longrightarrow x_0}\sup|U(x_0)-U(y_0)|\leq Cr.
		\end{eqnarray}
		Since this holds for any sufficiently small $r$, it follows that $U$ is continuous at $x_0$ and thus continuous on $D$ as well.
	\end{proof}
%	\begin{rem}
%		Looking at the proof, we easily see that the assumptions could be weakened in Lemma \ref{lem3.14}: $U \in W^{1,p}_0(D)$ would be sufficient and $r^{d-1}$ could be replaced in \ref{lar15} by $r^{d-2}\epsilon(r)$ with $\frac{\epsilon(r)}{r}$ integrable on $(0, 1)$.
%	\end{rem}
	\subsubsection{Lipschitz continuity of the optimal eigenfunction}
	\noindent\\
	\begin{thm}\label{thm3.16}
		Let $u$ be a solution of (\ref{p222}). Then $u$ is locally Lipschitz continuous on $D$.
	\end{thm}
\begin{prop}\label{Lem5.3}
	Under the assumptions of Theorem \ref{thm3.16},the solutionuof (\ref{p222}) satisfies:
	\begin{equation}\label{lar152}
		\ \Delta_p u +\Delta u  +\lambda_c |u|^{q-2}u\geq 0,  u \in L^{\infty}(D).
	\end{equation}
	There exists $C \geq 0$ such that, for all ball $B(x_0, r)$ such that $B(x_0, 2r) \subset D$
	\begin{eqnarray}\label{lar151}
		|\Delta_p |u|+\Delta |u||(B(x_0,r))\leq Cr^{d-1}.
	\end{eqnarray}
	Finally, $u$ is continuous on $D$.
\end{prop}
	We use the same technique as \cite{Briancon3} to show that $\Delta_p u +\Delta u  +\lambda_c |u|^{q-2}u\geq 0.$ This allows us to conclude further that $\Delta_p u +\Delta u  +\lambda_c |u|^{q-2}u$ is Radon measure.\\
	We start with a technical proposition.
	\begin{Lem}\label{pron2.1}
	Let $\Psi\in C^{\infty}_0(D)$ and $P\in W^{1,\infty}(\mathbb{R},\mathbb{R})$ with $P(0)=0.$ Then we have
	\begin{eqnarray*}
		\int_{D}P'(u(x))|\nabla u|^{p}\Psi+ \int_{D}P'(u(x))|\nabla u|^{2}\Psi&=&-\int_{D}P(u(x))|\nabla u|^{p-2}\nabla u\nabla \Psi-\int_{D}p(u(x))\nabla u\nabla \Psi\\
		& &+\lambda_c\int_{D}P(u(x))|u|^{q-2}u\Psi.		
	\end{eqnarray*}	
	\end{Lem}
	\begin{proof}[Proof of Lemma \ref{pron2.1}]
		Let $\Psi\in C^{\infty}_0(D)$ and $P\in W^{1,\infty}(\mathbb{R},\mathbb{R})$ with $P(0)=0.$ Let:
		\begin{equation*}
			v(x)=-\Psi(x)P(u(x)).
		\end{equation*}
		We have $v\in W^{1,p}_0(D).$ Indeed, $|P(u(x))|\leq \|p'\|_{\infty}| u(x)|,$ so that $v\in L^p(D),$ and
		\begin{equation*}
			|\nabla P(u(x))|=|P'(u(x))\nabla u(x)|\leq \|P'\|_{\infty}|\nabla u(x)|
		\end{equation*}
		so $\nabla v \in L^p(D)^d.$ Because $u(x)=0$ imply $v(x)=0,$ we get $|\Omega_{v}|\leq |\Omega_u|.$ We have
		\begin{eqnarray*}
			0=	\langle\Delta_pu+\Delta u+\lambda_c|u|^{q-1}u, v \rangle &=&-\int_{D}|\nabla u|^{p-2}\nabla u\nabla v -\int_{D}\nabla u\nabla v+\lambda_{c}\int_{D}|u|^{q-2}u v\\
			&=&\int_{D}P'(u(x))|\nabla u|^{p}\Psi+ \int_{D}P'(u(x))|\nabla u|^{2}\Psi +\int_{D}P(u(x))\nabla u\nabla \Psi\\
			& & +\int_{D}P(u(x))|\nabla u|^{p-2}\nabla u\nabla \Psi-\lambda_c\int_{D}P(u(x))|u|^{q-2}u\Psi	
		\end{eqnarray*} 
		Thus we have
		\begin{eqnarray*}
			\int_{D}P'(u(x))|\nabla u|^{p}\Psi+ \int_{D}P'(u(x))|\nabla u|^{2}\Psi&=&-\int_{D}P(u(x))|\nabla u|^{p-2}\nabla u\nabla \Psi-\int_{D}P(u(x))\nabla u\nabla \Psi\\
			& &+\lambda_c\int_{D}P(u(x))|u|^{q-2}u\Psi.		
		\end{eqnarray*}
	\end{proof}
	\begin{proof}[Proof of Proposition \ref{Lem5.3}]
		Until $u$ is replaced by $|u|$, we can assume that $u \geq0.$\\
		Let $\psi\in C_0^{\infty}(D),\psi\geq0.$	
		Let $p_n$ be defined by 
		\begin{equation}\label{equa2}
			p_n(r)=	\left
			\{
			\begin{array}{c r}
				
				0 & \mbox{if } \ r\leq 0 \\
				n r & \mbox{if} \ r \in [0,\frac{1}{n}] \\
				1 & \mbox{if} \ r\geq \frac{1}{n} .
			\end{array}
			\right.
		\end{equation}
		Applying proposition \ref{pron2.1}, we get 	
		\begin{eqnarray}\label{eqn17}
			\int_{D}p_n'(u(x))|\nabla u|^{p}\Psi+ \int_{D}p_n'(u(x))|\nabla u|^{2}\Psi&=&-\int_{D}p_n(u(x))|\nabla u|^{p-2}\nabla u\nabla \Psi-\int_{D}p(u(x))\nabla u\nabla \Psi\nonumber\\
			& &+\lambda_c\int_{D}p_n(u(x))|u|^{q-2}u\Psi.		
		\end{eqnarray}
		Note that $u p_n(u)\longrightarrow u^+=u, p_n(u)\nabla u\longrightarrow \nabla u$ in a non-decreasing manner as $n$ increases to $+\infty.$ Using $p_n'(u)|\nabla u|^p\geq0$ and $p_n'(u)|\nabla u|^2\geq0,$ we obtain in the limit that $\Delta_p u +\Delta u+\lambda_c |u|^{q-2}u\geq 0$ in the sense of the distributions in $D$. In particular $-\Delta_p u -\Delta u\leq\lambda_c |u|^{q-2}u $ on $D$, this implies that $\Delta_p u +\Delta u$ is a measure. Since $u\geq 0$ and $u\in W_0^{1,p}(\Omega_u)\cap L^{\infty}(\Omega_u)$ (see \cite{Rocard1}), we have $u\in L^{\infty}(D)$.\\
		By the proposition \ref{propvp}, $u$ is also a solution of the problem (\ref{p225}). We apply this inequality with $\nu=u+t\varphi, t>0, \varphi \in C^{\infty}_{0}(D)$ and $\varphi\geq0.$ Then
		
		\begin{eqnarray*}
			\dfrac{\frac{1}{p}\displaystyle\int_{D}| \nabla u| ^{p}dx+\frac{1}{2}\displaystyle\int_{D} |\nabla u| ^{2}dx}{\frac{1}{q}\displaystyle\int_{D}|u|^{q}dx}&\leq & \frac{1}{p}\displaystyle\int_{D}|\nabla v|^{p}+ \frac{1}{2}\int_{D}| \nabla v|^{2}\\ 
			 & &+\lambda_c\left[ 1-\frac{1}{q}\displaystyle\int_{D}|(u+t\varphi)|^{q}\right] +\mu\left[\Omega_v|- c\right]^+, v\in  W^{1,p}_{0}(D).
		\end{eqnarray*}
		We have:
		\begin{eqnarray*}
			\int_{D}|(u+t\varphi)|^{q}&\geq& \int_{D}|u|^{q};
		\end{eqnarray*}
		\begin{eqnarray*}
			\int_{D}|(\nabla u+t\nabla\varphi)|^{p}&=& \int_{D}|(\nabla u+t\nabla\varphi)|^{2}|\nabla u+t \nabla\varphi|^{p-2}\\
			&=&\int_{D}\left( |\nabla u|^2+2t\nabla u\nabla\varphi +t^2|\nabla\varphi|^{2}\right) |\nabla u+t \nabla\varphi|^{p-2}\\
			&\leq&\int_{D}\left( |\nabla u|^2+2t\nabla u \nabla\varphi +t^2|\nabla\varphi|^{2}\right)\left( 2^{p-2}|\nabla u|^{p-2}+2^{p-2}t^{p-2}| \nabla\varphi|^{p-2}\right) \\
			&\leq& \int_{D}( 2^{p-2}|\nabla u|^{p}+2^{p-1}t|\nabla u|^{p-2}\nabla u \nabla\varphi \\
			& &+2^{p-2}t^2|\nabla u|^{p-2}|\nabla\varphi|^{2}
			+2^{p-2}t^{p-2}|\nabla u|^{2}|\nabla \varphi|^{p-2}\\
			& &+2^{p-1}t^{p-1}|\nabla\varphi|^{p-2}\nabla\varphi \nabla u
			+2^{p-2}t^{p}|\nabla\varphi|^{p})
		\end{eqnarray*}	and
		\begin{equation*}
			\int_{D}|(\nabla u+t\nabla\varphi)|^{2}
		=\int_{D}\left( |\nabla u|^2+2t\nabla u\nabla\varphi +t^2|\nabla\varphi|^{2}\right)	
		\end{equation*}
		which gives
		\begin{align*}
				\dfrac{\frac{1}{p}\displaystyle\int_{D}|\nabla u| ^{p}dx+\frac{1}{2}\displaystyle\int_{D}|\nabla u| ^{2}dx}{\frac{1}{q}\displaystyle\int_{D}|u|^{q}dx}
			&\leq  \frac{1}{p}\int_{D}|\nabla (u+t\varphi)|^{p}+ \frac{1}{2}\int_{D}| \nabla (u+t\varphi)|^{2} \\
			& + \lambda_c\left[ 1-\frac{1}{q}\int_{D}|u+t\varphi|^q\right]^+ +\mu\left[|\Omega_{(u+t\varphi)}|- c\right]^+\\
			&\leq\frac{2^{p-2}}{p} \int_{D}\left( |\nabla u|^{p}+2t|\nabla u|^{p-2}\nabla u \nabla\varphi\right)\\
			&  +t^2|\nabla u|^{p-2}|\nabla\varphi|^{2}
			+\frac{2^{p-2}}{p} \int_{D}t^{p-2}|\nabla u|^{2}|\nabla \varphi|^{p-2}\\
			& +\frac{2^{p-1}t^{p-1}}{p} \int_{D}|\nabla\varphi|^{p-2}\nabla\varphi \nabla u+\frac{2^{p-2}t^{p}}{p} \int_{D}|\nabla\varphi|^{p}\\
			& + \frac{1}{2}\int_{D}\left( |\nabla u|^2+2t\nabla u\nabla\varphi +t^2|\nabla\varphi|^{2}\right)\\
			&  + \lambda_c\left[ 1- \frac{1}{q}\int_{D}\left( |u|^{q}\right)\right]
			+\mu\left[|\Omega_{(u+t\varphi)}|- c\right]^+	
			\end{align*}
		\begin{align*}
			\dfrac{\frac{1}{p}\displaystyle\int_{D}|\nabla u| ^{p}dx+\frac{1}{2}\displaystyle\int_{D}|\nabla u| ^{2}dx}{\frac{1}{q}\displaystyle\int_{D}|u|^{q}dx}&\leq \lambda_c+\frac{1}{p}\int_{D}|\nabla u| ^{p}dx+\frac{1}{2}\int_{D}|\nabla u| ^{2}dx - \frac{\lambda_c}{q}\int_{D} |u|^{q}\\
			& +\frac{2^{p-2}-1}{p} \int_{D}|\nabla u|^{p}+ \frac{2^{p-1}t}{p} \int_{D} |\nabla u|^{p-2}\nabla u \nabla\varphi\\
			& +\frac{2^{p-2}t^2}{p} \int_{D}|\nabla u|^{p-2}|\nabla\varphi|^{2}
			 +\frac{2^{p-2}t^{p-2}}{p} \int_{D}|\nabla u|^{2}|\nabla \varphi|^{p-2}\\
			 &+\frac{2^{p-2}t^{p-1}}{p} \int_{D}|\nabla\varphi|^{p-2}\nabla\varphi \nabla u
			 +\frac{2^{p-2}t^{p}}{p} \int_{D}|\nabla\varphi|^{p}
			 + t\int_{D}\nabla u\nabla\varphi\\ &+\frac{t^2}{2}\int_{D}|\nabla\varphi|^{2} +\mu\left[\Omega_{(u+t\varphi)}|- c\right]^+,			
		\end{align*}
		
		and we have
		\begin{eqnarray*}
			0&\leq&\frac{2^{p-2}-1}{p} \int_{D}|\nabla u|^{p}+ \frac{2^{p-1}t}{p} \int_{D} |\nabla u|^{p-2}\nabla u \nabla\varphi +\frac{2^{p-2}t^2}{p} \int_{D}|\nabla u|^{p-2}|\nabla\varphi|^{2}\\
			& &+\frac{2^{p-2}t^{p-2}}{p} \int_{D}|\nabla u|^{2}|\nabla \varphi|^{p-2}+\frac{2^{p-2}t^{p-1}}{p} \int_{D}|\nabla\varphi|^{p-2}\nabla\varphi \nabla u+\frac{2^{p-2}t^{p}}{p} \int_{D}|\nabla\varphi|^{p}\\
			& &+ t\int_{D}\nabla u\nabla\varphi +\frac{t^2}{2}\int_{D}|\nabla\varphi|^{2} +\mu\left[\Omega_{(u+t\varphi)}|- c\right]^+\\
			&\leq&  \frac{2^{p-1}t}{p} \left( \int_{D} |\nabla u|^{p-2}\nabla u \nabla\varphi+\int_{D}\nabla u\nabla\varphi\right)+\frac{2^{p-1}}{p} \int_{D}|\nabla u|^{p}  +\frac{2^{p-2}t^2}{p} \int_{D}|\nabla u|^{p-2}|\nabla\varphi|^{2}\\
			& &+\frac{2^{p-2}t^{p-2}}{p} \int_{D}|\nabla u|^{2}|\nabla \varphi|^{p-2}+\frac{2^{p-2}t^{p-1}}{p} \int_{D}|\nabla\varphi|^{p-2}\nabla\varphi \nabla u+\frac{2^{p-2}t^{p}}{p} \int_{D}|\nabla\varphi|^{p}\\
			& & +\frac{t^2}{2}\int_{D}|\nabla\varphi|^{2} +\mu\left[|\Omega_{(u+t\varphi)}|- c\right]^+\\
		\end{eqnarray*}
		
		for $t\leq\frac{p}{2^{p-1}}$ we have	
		\begin{eqnarray}\label{eq3.22}
			0&\leq&  \int_{D} |\nabla u|^{p-2}\nabla u \nabla\varphi+\int_{D}\nabla u\nabla\varphi+\frac{2^{p-1}}{p} \int_{D}|\nabla u|^{p}  +\frac{2^{p-2}t^2}{p} \int_{D}|\nabla u|^{p-2}|\nabla\varphi|^{2}\nonumber\\
			& &+\frac{2^{p-2}t^{p-2}}{p} \int_{D}|\nabla u|^{2}|\nabla \varphi|^{p-2}+\frac{2^{p-2}t^{p-1}}{p} \int_{D}|\nabla\varphi|^{p-2}\nabla\varphi \nabla u+\frac{2^{p-2}t^{p}}{p} \int_{D}|\nabla\varphi|^{p} +\frac{t^2}{2}\int_{D}|\nabla\varphi|^{2}\\
			& & +\mu\left[|\Omega_{(u+t\varphi)}|- c\right]^+\nonumber
		\end{eqnarray} we have
		\begin{eqnarray*}
			-\int_{D} |\nabla u|^{p-2}\nabla u \nabla\varphi-\int_{D}\nabla u\nabla\varphi&\leq& \frac{2^{p-1}}{p} \int_{D}|\nabla u|^{p}  +\frac{2^{p-2}t^2}{p} \int_{D}|\nabla u|^{p-2}|\nabla\varphi|^{2}\\
			& &+\frac{2^{p-2}t^{p-2}}{p} \int_{D}|\nabla u|^{2}|\nabla \varphi|^{p-2}\nonumber
			+\frac{2^{p-2}t^{p-1}}{p} \int_{D}|\nabla\varphi|^{p-2}\nabla\varphi \nabla u\\
			& & +\frac{2^{p-2}t^{p}}{p} \int_{D}|\nabla\varphi|^{p} +\frac{t^2}{2}\int_{D}|\nabla\varphi|^{2} +\mu|\Omega_{\varphi}|
		\end{eqnarray*}
		when $t\rightarrow 0$ we obtain 
		\begin{eqnarray*}
			\langle\Delta_pu+\Delta u+\lambda_{c}|u|^{q-1},\varphi \rangle
			&\leq& \frac{2^{p-2}}{p} \int_{D}|\nabla u|^{p}+\lambda_{c}\int_{D}|u|^{q-1}\varphi +\mu|\Omega_{\varphi}|
		\end{eqnarray*}
		Furthermore, we have,
		\begin{equation*}
			\int_{D}|\nabla u|^{p}dx+\int_{D}|\nabla u|^{2}\leq\lambda_c\int_{D}| u|^{q}
		\end{equation*}
		so $\int_{D}|\nabla u|^{p}dx\leq \lambda_c\int_{D}| u|^{q}$ so we have
		\begin{eqnarray}\label{eqq3.24}
			\langle\Delta_pu+\Delta u+\lambda_{c}|u|^{q-2}u,\varphi \rangle
			&\leq& \frac{2^{p-2}\lambda_c}{p} \int_{D}|u|^{q}+\lambda_{c}\int_{D}|u|^{q-2}u\varphi +\mu|\Omega_{\varphi}|
		\end{eqnarray}
		
		Now let $x_0\in D$ be such that $B(x_0,2r)\subset D$ and let $\varphi \in C^{\infty}_{0}(B(x_0,2r))^+$ be as in \ref{eq3.16}. Also using $u\in L^{\infty},$ we deduce that
		
		\begin{eqnarray}
			|\Delta_p u+\Delta u|B(x_0,r)\leq (\Delta_p u+\Delta u+\lambda_{c}|u|^{q-1} )B(x_0,r)+\lambda_{c}\int_{D}|u|^{q-1}\leq Cr^{d-1},
		\end{eqnarray}
		for $r\leq1$, hence the estimate (\ref{lar151}).
		By the lemma \ref{lem3.14}, we conclude that $u$ is continuous.
	\end{proof}
	\begin{proof}[Proof of the theorem \ref{thm3.16}]
		Until $u$ is replaced by $|u|$, we can assume that $u \geq0$. By the lemma \ref{Lem5.3}, $u$ is continuous on $D$. Therefore, we will have $ -\Delta_p u -\Delta u = \lambda_c |u|^{q-2}u$ on the open set $\omega = [u> 0]$ (see Remark \ref{rem3.3}).  Since $u$ is a eigenfunction on $\omega$ then we   $u\in L^{\infty}(\omega)$ and $\lambda_c |u|^{q-2}u \in L^{\infty}(\omega)$.	
		Thus, with the lemma \ref{Lem5.3}, the assumptions of the proposition \ref{prop11} are satisfied and local Lipschitz continuity on $D$ follows.
		
	\end{proof}
	Let us now state a corollary of the Theorem \ref{thm3.16} for the initial real shape optimization problem (\ref{glm}).
	
	\begin{Cor}\label{cor5.1}
		Suppose that $D\subset \mathbb{R}^d$ is open set and of finite measure. Then there exists an open set $\Omega^*\subset D$ which is a solution of (\ref{glm}). Moreover, for any (quasi-open) solution $\Omega^*$ of (\ref{glm}), $u_{\Omega^*}$ is locally Lipschitz continuous on D. If moreover D is connected, then all solutions $\Omega^*$ of (\ref{glm}) are open.
	\end{Cor}
	\begin{rem}
		If $D$ is not connected, then it may happen that $\Omega^*$ is not open. However, $u_{\Omega^*}$ is always locally Lipschitz continuous. 
	\end{rem}
	\begin{proof}[Proof of Corollaire \ref{cor5.1}]
		If $D$ is of finite measure, as we have already seen, the problem (\ref{p222}) has a solution $u$. By the theorem (\ref{thm3.16}), it is locally Lipschitz continuous on $D$. In particular $\Omega_u$ is open. If $|\Omega_u|=c$ then $\Omega^*:=\Omega_u$ is an open solution of problem (\ref{p222}). If $|\Omega_u|<c$, then any open $\Omega^*$ satisfying $ \Omega_u\subset \Omega^*\subset  D,\ |\Omega^*|=c$ is also a solution since by monotonicity we have: $\lambda_{1}(\Omega^*)\leq \lambda_{1}(\Omega_u)$ (and there exist such $\Omega^*$ as for example $\Omega^*:= \Omega_u\cup B(x_0,r)\cap D$ where $x_0 \in D$ and $r$ is chosen such that $|\Omega^*|=c$).
		
		Let $\Omega^*$ now be a solution of (\ref{glm}). Then, by the proposition \ref{prop3.4}, $u_{\Omega^*}$ is a solution of the minimization problem (\ref{p222}). By the theorem \ref{thm3.16}, it is locally Lipschitz continuous in $D$. As proved in remark \ref{rem3.6}, if $D$ is connected, then $\Omega^* = [u_{\Omega^*} >0]$. Therefore $\Omega^*$ is open.
	\end{proof}
	\section{Regularity of the boundary}
	Here, we are interested in the regularity of $\partial\Omega^*$ where 
	\begin{equation*}
	\Omega^*=\Omega_u=\left\lbrace x \in D, u_{\Omega^*}>0\right\rbrace 	
	\end{equation*} 	
\begin{Def}
	Let $O$ denotes an open subset of $\mathbb{R}^d$. A function $f\in L^1(O)$ has bounded variation in $O$ if 
	\begin{eqnarray}
		\sup \left\lbrace \int_{O} f div \varphi dx/ \varphi \in C^1_c(O,\mathbb{R}^d), \ |\varphi|<1\right\rbrace < \infty.
	\end{eqnarray}
We write 
\begin{equation*}
	BV(O)
\end{equation*}
to denote the space of functions of bounded variation.
\end{Def}	
\begin{Def}
	An $\mathcal{L}^d-$measurable subset $E\subset \mathbb{R}^d$ has finite perimeter in $O$ if 
	\begin{equation*}
		\chi_E\in 	BV(O).
	\end{equation*} 
\end{Def}	
It is convenient to introduce also local versions of the above concepts:
\begin{Def}
	Let $O$ denotes an open subset of $\mathbb{R}^d$. A function $f\in L^1(O)$ has locally bounded variation in $O$ if for each open set $V\subset\subset O$
	\begin{eqnarray}
		\sup \left\lbrace \int_{V} f\ div \varphi dx/ \varphi \in C^1_c(V,\mathbb{R}^d), \ |\varphi|<1\right\rbrace < \infty.
	\end{eqnarray}
	We write 
	\begin{equation*}
		BV_{loc}(O)
	\end{equation*}
	to denote the space of such functions .
\end{Def}	
\begin{Def}
	An $\mathcal{L}^d-$measurable subset $E\subset \mathbb{R}^d$ has locally finite perimeter in $O$ if
	\begin{equation*}
		\chi_E\in 	BV_{loc}(O).
	\end{equation*}
\end{Def}
\begin{thm}[Structure Theorem for $BV_{loc}$ Functions]
Let $f\in BV_{loc} (O).$ Then there exists a Radon measure $\mu$ on $O$ and a $\mu-$measurable function $\sigma : O \rightarrow	\mathbb{R}^d$ such that 
\begin{enumerate}
	\item[(i)] $|\sigma(x)|=1$ $\mu\ a.e$ and 
	\item [(ii)]
	\begin{equation*}
		\int_{O} f div \varphi dx=-\int_{O} \varphi .\sigma d\mu
	\end{equation*}
for all $\varphi\in C^1_c(O,\mathbb{R}^d).$
\end{enumerate}
\end{thm}
\begin{proof}
	See \cite{Evans}.
\end{proof}
\textbf{Notation}\\
\begin{enumerate}
	\item If $ f \in BV_{loc}(O),$ we will henceforth write $\| Df\| $ for the measure $\mu$ and 
	\begin{equation*}
		[ Df]\equiv \|Df\| \lfloor \sigma.
	\end{equation*}
\item Similarly, if $f=\chi_E$ and $E$ is a set of locally finite perimeter in $O,$ we will hereafter write $	\| \partial E\|$
for the measure, and $\nu_E\equiv -\sigma.$\\
Consequently, 
\begin{equation*}
	\int_{E}  div \varphi dx=-\int_{O} \varphi . \nu_E d \| \partial E\|
\end{equation*}
for all $\varphi\in C^1_c(O,\mathbb{R}^d).$
\end{enumerate}
We hereafter assume $E$ is a set of locally finite perimeter in $\mathbb{R}^d.$
\begin{Def}
	Let $x\in \mathbb{R}^d.$ We say $x\in \partial^* E,$ the reduced boundary of E, if 
	\begin{enumerate}
		\item 
		\begin{equation*}
			\| \partial E\| B(x,r)>0
		\end{equation*}
	for all $r>0,$
	\item 
	\begin{equation*}
		\lim\limits_{r\rightarrow 0} \fint_{ B(x,r)} \nu_E d\| \partial E\|= \nu_E(x),
	\end{equation*}
	and 
	\item 
	\begin{equation*}
		|\nu_E(x)|=1.
	\end{equation*}	 
	\end{enumerate}
\end{Def}
\begin{rem}
	If $E$ is an open set with $C^1$ boundary, then $\partial^* E=\partial E$ and the measure-theoretic outer unit normal $\nu_E$ coincides with classical notion of outer unit normal.
\end{rem}
We shall prove the following theorem:			
	\begin{thm}\label{theo1.2}
		Suppose that $D$ is open, bounded and connected. Then any solution of (\ref{glm}) satisfies :
		\begin{enumerate}
			\item 
			There is a Radon-Nikodym density $f$ of $\mu_1$ with respect to $\mathcal{H}^{d-1}$ such that
			\begin{eqnarray*}
				\Delta_pu_{\Omega^*}+\Delta u_{\Omega^*}+\lambda_1|u_{\Omega^*}|^{q-2}u_{\Omega^*}= f \mathcal{H}^{d-1}\lfloor\partial\Omega^*. 
			\end{eqnarray*}
			where $\mu_1=\Delta_pu_{\Omega^*}+\Delta u_{\Omega_u}+\lambda_1|u_{\Omega_u}|^{q-2}u_{\Omega_u}$
			\item $\Omega^*$ has locally finite perimeter in $D$ and
			\begin{eqnarray}
				\mathcal{H}^{d-1}((\partial\Omega^*\backslash\partial^*\Omega^*)\cap D)=0,
			\end{eqnarray}
			where $\mathcal{H}^{d-1}$ is the Hausdorff measure of dimension $d - 1$. 
%			and $\partial^*\Omega^*$ is the reduced boundary (in the sense of finite perimeter sets, see \cite{Evans} or \cite{Giusti}). 
		\end{enumerate}
	\end{thm}
	
	\begin{Lem}\label{theon2.2}
		\begin{equation*}
			\mu_1=\Delta_p u_{\Omega^*} +\Delta u_{\Omega^*}+ \lambda_{c}|u_{\Omega^*}|^{q-2}u_{\Omega^*} 
		\end{equation*}
		is a positive Radon measure.
	\end{Lem}
	\begin{proof}[Proof of Lemma \ref{theon2.2}]
		In lemma \ref{Lem5.3}, we have shown that $\mu_1\geq 0$. It is therefore a Radon measure
	\end{proof}
	We now show that $\mu_1$ is absolutely continuous with respect to the Hausdorff mesure $\mathcal{H}^{d-1}$ in $D$. More precisely we have the following
	\begin{prop}\label{propn2.3}
		Let $u=u_{\Omega^*}$ be a solution of (\ref{p222}).There exists $C>0$ such that, for every ball $B(x,r)\subset D, r\leq 1,$ we have 
		\begin{equation*}
			\mu_1(B(x,r))\leq C r^{d-1}.
		\end{equation*}	
	\end{prop}
	\begin{proof}[Proof of Proposition \ref{propn2.3} ]
		We use the same notations as in the proof of Theorem\ref{theon2.2}.Let $u=u_{\Omega^*}\in W^{1,\infty}_{loc} (D)$ be a solution of (\ref{p222}). Let $B(x,r)\subset D$, then
		\begin{eqnarray*}
			& &\int_{B(x,r)}div\left( |\nabla u_{\Omega^*}|^{p-2}p_n(u_{\Omega^*})\nabla u_{\Omega^*}   +p_n(u_{\Omega^*})\nabla u\right) +\lambda_c |u|^{q-1}p_n(u)\\
			&=&\int_{\partial B(x,r)}|\nabla u_{\Omega^*}|^{p-2}p_n(u_{\Omega^*})\nabla u_{\Omega^*}.n+\int_{\partial B(x,r)}p_n(u_{\Omega^*})\nabla u_{\Omega^*}.n+ \lambda_c\int_{B(x,r)}|u|^{q-1}p_n(u)\\
			&\leq & \|\nabla u_{\Omega^*}\|^{p-1}_{L^{\infty}(B(x,r))} d\omega_dr^{d-1}+\|\nabla u_{\Omega^*}\|_{L^{\infty}(B(x,r))} d\omega_dr^{d-1}+\|u_{\Omega^*}\|_{L^{\infty}(B(x,r))}\omega_d r^d\\
			&\leq& C r^{d-1},
		\end{eqnarray*}	
		for $r\leq 1.$ Because $div\left( |\nabla u_{\Omega^*}|^{p-2}p_n(u_{\Omega^*})\nabla u_{\Omega^*}   +p_n(u_{\Omega^*})\nabla u\right) +\lambda_c |u|^{q-1}p_n(u)$ converges weakly in the sense of Radon measures to $\Delta_p u_{\Omega^*} +\Delta u_{\Omega^*}+ \lambda_{c}|u_{\Omega^*}|^{q-2}u_{\Omega^*}$ we deduce:
		\begin{equation*}
			\mu_1(B(x,r))\leq \lim\limits_{n\rightarrow \infty}\inf\int_{B(x,r)}div\left( |\nabla u_{\Omega^*}|^{p-2}p_n(u_{\Omega^*})\nabla u_{\Omega^*}   +p_n(u_{\Omega^*})\nabla u\right) +\lambda_c |u|^{q-1}p_n(u)\leq C r^{d-1}.
		\end{equation*}
	\end{proof}	
\begin{thm}[Lebesgue decomposition]\cite{Cedric}\label{rod2}
	Let $(X, \mathcal{T}, \nu)$ be a $\sigma-$finite measured space, and $\mu$ a $\sigma-$finite signed measure on $(X, \mathcal{T})$. Then there is a measurable function $f: X\rightarrow \mathbb{R}$ and a measure signed by $\mu_s$ such that
	    \begin{equation*}
\mu= f\nu+\mu_s, \ \mu_s\perp\nu.
	    \end{equation*}
	
	 The $\mu_s$ measure is unique, and the $f$ function is unique to within a modification on a set of $\nu-$zero measures; in particular, the $f\nu$ measure is unique.
\end{thm}
The measure $\mu_s$ appearing in Theorem \ref{rod2} is called the singular part of $\mu$ with respect to $\nu$. The function $f$ appearing in this same statement is called the Radon-Nikodym density of $\mu$ with respect to $\nu$, and is denoted by
\begin{equation*}
	f=\frac{d \mu}{ d\nu}
\end{equation*}
this function is well defined to within a $\nu-$negligible  set.
	\begin{Cor}[Radon Nikodym's Theorem]\cite{Cedric}\label{cor6.1}
		Let $(X, \mathcal{T}, \nu)$ be a $\sigma-$finite measured space, and $\mu$ a $\sigma-$finite signed measure. Then the following two statements are equivalent:	
		\begin{enumerate}
			\item [(i)] $\mu$ does not load any $\nu-$negligible set; i.e.
			\begin{equation}\label{eqr78}
				\forall N \in \mathcal{T}, \ \left[ \nu(N)=0\Longrightarrow \mu(N)=0\right]  .
			\end{equation}
			\item [(ii)] $\mu$ is $\nu-$measurable; i.e.
			\begin{equation}
				\exists f \in L(X,\nu);\ \mu=f\nu.
			\end{equation}
		\end{enumerate}
	\end{Cor}
	\begin{proof}[Proof of Corollaire \ref{cor6.1} ]
		By Hahn's Decomposition Theorem, there exist disjoint measurable parts $S_+$ and $S_-$ such that $\mu_+$ and $\mu_-$ are concentrated on $S_+$ and $S_-$ respectively. If $N$ is a $ \nu-$negligible set, so is $S_+\cap N$ and $S_-\cap N$, and so the assumption implies  $\mu(S_+\cap N)=\mu(S_-\cap N)=0$; so it suffices to show the corollary in the case where $\mu$ is an unsigned measure, which we will assume.\\
		Now let $N$ be a $\nu-$negligible set on which $\mu_s$ is concentrated. For any $A\subset N$ we have $\nu(A)=0$, hence $\mu(A)=0$ by hypothesis, and of course $f\nu(A)=0$; we deduce $\mu_s(A)=0$; $\mu_s$ is thus the null measure, and the conclusion follows. 	
	\end{proof}
	\begin{rem}
		A very often used terminology is to say that $\mu$ verifying (\ref{eqr78}) is absolutely continuous with respect to $\nu$.
	\end{rem}
	
	\begin{proof}[Proof of the first point of the Theorem \ref{theo1.2}]
		
		In proposition \ref{propn2.3}, we have shown that $\mu_1$ is absolutely continuous with respect to the Hausdorff mesure $\mathcal{H}^{d-1}$ in $D$. So we have
		\begin{eqnarray*}
			\Delta_pu_{\Omega^*}+\Delta u_{\Omega^*}+\lambda_1|u_{\Omega^*}|^{q-2}u_{\Omega^*}=f\mathcal{H}^{d-1}\lfloor\partial\Omega^*. 
		\end{eqnarray*}
		where $f=\frac{d \mu_1}{d\mathcal{H}^{d-1}}$.
	\end{proof}
	For $w\in W^{1,p}_0(D)$, we denote 
	\begin{equation*}
		J(w)=\frac{1}{p}\int_{D}|\nabla w|^{p}+ \frac{1}{2}\int_{D}| \nabla w|^{2} - \frac{\lambda_c}{q}\int_{D}|w|^{q}.	
	\end{equation*}
	It is easy to see that the optimization problem (\ref{p222}) is equivalent to the following optimization problem:
	\begin{equation}\label{equation 6}
		J(u)\leq J(w) \ \mbox{for all}\ w\in W^{1,p}_0(D)\ \mbox{with}\ |\Omega_{w}|\leq c
	\end{equation}
	In the following we will only consider solutions of (\ref{equation 6}).
	The Euler-Lagrange equation of the minimization problem is given by the following Lemma.
	\begin{Lem}[Euler–Lagrange equation]\label{Lema2.1}
		Let $u$ be a solution of (\ref{equation 6}). Then there exists $\Lambda>0$ such that, for all $\Phi \in C^{\infty}_0(D,\mathbb{R}^d)$
		\begin{eqnarray}\label{euler1}
		& &	\int_{D}|\nabla u|^{p-2}(D\Phi\nabla u,\nabla u)+\int_{D}(D\Phi\nabla u,\nabla u)-\frac{1}{p}\int_{D}|\nabla u|^{p}\nabla.\Phi- \frac{1}{2}\int_{D}| \nabla u|^{2}\nabla.\Phi\nonumber\\
		& &+\frac{\lambda_c}{q}\int_{D}|u|^{q}\nabla.\Phi=\Lambda\int_{\Omega_u}\nabla.\Phi
		\end{eqnarray}
	\end{Lem}
	\begin{proof}[Proof of the Lemma \ref{Lema2.1} ]
		We start by a general remark that will be useful in the rest of the paper. If $v\in W^{1,p}_0(D)$ and if $\Phi \in C^{\infty}_0(D,\mathbb{R}^d)$ 
		we define $v_t(x) = v(x +t\Phi)$; therefore, for $t$ small enough, $v_t\in  W^{1,p}_0(D).$ A simple calculus gives (when t goes to 0),
		\begin{equation*}
			|\Omega_{v_t}|=|\Omega_v|- t\int_{\Omega_v}\nabla.\Phi +o(t),
		\end{equation*}
	\begin{eqnarray*}
			& & J(v_t)= J(v) \\
			& & +t\left( 	\int_{D}|\nabla v|^{p-2}(D\Phi\nabla v,\nabla v)+\int_{D}(D\Phi\nabla v,\nabla v)-\frac{1}{p}\int_{D}|\nabla v|^{p}\nabla.\Phi- \frac{1}{2}\int_{D}| \nabla v|^{2}\nabla.\Phi 	+\frac{\lambda_c}{q}\int_{D}|v|^{q}\nabla.\Phi\right) \\
			& &+o(t).
	\end{eqnarray*}
	
		Now we apply this with $v=u$ and $\Phi$ such that $\int_{\Omega_u}\nabla.\Phi>0.$  Such a $\Phi$ exists, otherwise we would get, using that $D$
		is connected, $\Omega_u = D$ or $\emptyset$ a.e. We have $|\Omega_{u_t}| < |\Omega_u|$ for $t > 0$ small enough and, by minimality,
		\begin{eqnarray*}
			J(u)&\leq& J(u_t)\nonumber\\
			&=&J(u) + t\left( 	\int_{D}|\nabla u|^{p-2}(D\Phi\nabla u,\nabla u)+\int_{D}(D\Phi\nabla u,\nabla u)-\frac{1}{p}\int_{D}|\nabla u|^{p}\nabla.\Phi- \frac{1}{2}\int_{D}| \nabla u|^{2}\nabla.\Phi \right)\\ 
			& &+ t\left(- \frac{1}{2}\int_{D}| \nabla u|^{2}\nabla.\Phi+ \frac{\lambda_c}{q}\int_{D}|u|^{q}\nabla.\Phi\right) +o(t),
		\end{eqnarray*}
		and so,
		\begin{equation}\label{equation 10}
			\int_{D}|\nabla u|^{p-2}(D\Phi\nabla u,\nabla u)+\int_{D}(D\Phi\nabla u,\nabla u)-\frac{1}{p}\int_{D}|\nabla u|^{p}\nabla.\Phi- \frac{1}{2}\int_{D}| \nabla u|^{2}\nabla.\Phi + \frac{\lambda_c}{q}\int_{D}|u|^{q}\nabla.\Phi\geq0.
		\end{equation}
		Now, we take $\Phi$ with $\int_{\Omega_u}\nabla.\Phi=0.$ Let $\Phi_1$ such that $\int_{\Omega_u}\nabla.\Phi_1=1.$ Writing (\ref{equation 10}) with $\Phi+\eta\Phi_1$ and letting $\eta$ goes to $0$, we get (\ref{equation 10}) with this $\Phi$ and, using $-\Phi$, we get (\ref{equation 10}) with an equality instead of the inequality. For a general $\Phi$, we use this equality with $\Phi-\Phi_1(\int_{\Omega_u}\nabla.\Phi)$(we have
		$\int_{\Omega_u}\nabla.(\Phi-\Phi_1(\int_{\Omega_u}\nabla.\Phi))=0$) and we get the result with 
		\begin{equation*}
			\Lambda=\int_{D}|\nabla u|^{p-2}(D\Phi_1\nabla u,\nabla u)+\int_{D}(D\Phi_1\nabla u,\nabla u)-\frac{1}{p}\int_{D}|\nabla u|^{p}\nabla.\Phi_1- \frac{1}{2}\int_{D}| \nabla u|^{2}\nabla.\Phi_1 + \frac{\lambda_c}{q}\int_{D}|u|^{q}\nabla.\Phi_1\geq0
		\end{equation*}
		using (\ref{equation 10}).\\
		It remains to show that $\Lambda\neq0.$ For that let us reason by the absurd by supposing that $\Lambda=0.$ By replacing $\Lambda=0$ in (\ref{euler1}), we have for all $\Phi \in C^{\infty}_0(D,\mathbb{R}^d)$
		\begin{equation}\label{euler2}
			\int_{D}\left( |\nabla u|^{p-2}+1\right)  (D\Phi\nabla u,\nabla u)=\frac{1}{p}\int_{D}|\nabla u|^{p}\nabla.\Phi+ \frac{1}{2}\int_{D}| \nabla u|^{2}\nabla.\Phi - \frac{\lambda_c}{q}\int_{D}|u|^{q}\nabla.\Phi
		\end{equation}
		Let $x\in D$ and $r>0$ such that $B(x,r)\subset D$, let $\Phi \in C^{\infty}_0(B(x,r),\mathbb{R}^d)$, $u$ being in $W^{1,\infty}_{loc}(D)$ we have
		\begin{equation*}
			\frac{1}{p}\int_{B(x,r)}|\nabla u|^{p}\nabla.\Phi+ \frac{1}{2}\int_{B(x,r)}| \nabla u|^{2}\nabla.\Phi - \frac{\lambda_c}{q}\int_{B(x,r)}|u|^{q}\nabla.\Phi=0.
		\end{equation*}	
		Thus we have:
		
		\begin{equation*}
			\int_{B(x,r)}\left( |\nabla u|^{p-2}+1\right)  (D\Phi\nabla u,\nabla u)=0 \ \forall \Phi \in C^{\infty}_0(B(x,r),\mathbb{R}^d),
		\end{equation*}
		which implies 
		\begin{equation*}
			\int_{B(x,r)} (D\Phi\nabla u,\nabla u)=0 \ \forall \Phi \in C^{\infty}_0(B(x,r),\mathbb{R}^d).
		\end{equation*}
		So we have $\nabla u=0$ on $B(x,r)$ which implies that $\nabla u=0$ a.e on $D$. Indeed if this were not the case, there would exist $E\subset D$, $|E|\neq 0$ such that $\nabla u\neq0$ on $E.$ It is thus clear to see that $E\subset \Omega_u$. Let $x\in E$ and $r>0$ as small as it may be, we have $B(x,r)\cap E\subset B(x,r)\cap \Omega_u $ and $|B(x,r)\cap \Omega_u|\neq 0,$ and $\nabla u=0$ on $B(x,r)\cap \Omega_u$ so $\nabla u=0$ on $B(x,r)\cap E.$ Thus $\nabla u=0$ on $B(x,r)\cap E$,$  \forall x\in E$, which contradicts the hypothesis. Therefore $\nabla u=0$ a.e on $D$ and as  $u\in W^{1,p}_0(D)$ then $u=0$ a.e on $D$, which contradicts the fact that $u$ is a non-zero eigenfunction. In sum $\Lambda> 0.$ 
	\end{proof}
	In the rest of this paper, we will be mainly interested in the solutions of (\ref{equation 6} ) .
	We will now see that $\Omega^*=\Omega_{u}$ has locally finite perimeter in $D$.
	\begin{thm}\label{theon2.4}
		Let $u$ be a solution (\ref{equation 6}). Then $\Omega_u=\{u>0\}$ has locally finite perimeter in $D$. Moreover, there exist constants $C, C_1, r_0$ depending on the data such that for every $B(x,r)\subset D$ with $r\leq r_0,$
		
		\begin{equation*}
			P(\Omega_u,B(x,r))\leq C \mu_1(B(x,r))\leq C_1r^{d-1}
		\end{equation*}
		where $P(\Omega_u,B(x,r))$ is the relative perimeter of $\Omega_u$  with respect to a set $B(x,r)$.
	\end{thm}
	The proof of Theorem \ref{theon2.4} will require the following lemma which says in a very weak sense that "$\Lambda=\frac{p-1}{p}|\nabla u|^p+\frac{1}{2}|\nabla u|^2$" on $\partial\{u>0\}$. 
	\begin{Lem}\label{Lemn2.5}
		Let $u$ be a solution (\ref{equation 6}). Then for every $\varphi \in C^{\infty}_0(D, \mathbb{R}^d)$ we have:
		\begin{eqnarray*}
			& &\lim\limits_{\epsilon\rightarrow 0}		\frac{1}{\epsilon}\int_{\{0<u<\epsilon\}}\langle \varphi. \nabla u \rangle\left(\frac{1-p}{p}|\nabla u|^p-\frac{1}{2}|\nabla u|^2
			+\Lambda\right) =0
		\end{eqnarray*}
	\end{Lem}
	\begin{proof}[Proof of Lemma \ref{Lemn2.5}]
		Let $\phi \in C^{\infty}_0(D, \mathbb{R}^d)$ and $\Psi_{\epsilon}(s)=\max (0, 1-\frac{s}{\epsilon}).$ We write Euler-Lagrange’s equation (\ref{euler1}) with $\Phi=\varphi\Psi_{\epsilon}(u)\in W^{1,\infty}(D)$	which has compact support. We study each term:
		\begin{eqnarray*}
			\int_{\{u>0\}} div \Phi	&=&\int_{\{u>0\}}\Psi_{\epsilon}(u)div(\varphi)+\int_{\{0<u<\epsilon\}}\langle \varphi.\nabla u\rangle \Psi_{\epsilon}'(u)\\
			&=&\int_{\{u>0\}}\Psi_{\epsilon}(u)div(\varphi)-\frac{1}{\epsilon}\int_{\{0<u<\epsilon\}}\langle \varphi.\nabla u\rangle
		\end{eqnarray*}
		Since $\Psi_{\epsilon}(u)\chi_{\{u>0\}}$ converge to $0$ a.e when $\epsilon$ goes to $0$, by dominated convergence, the first term goes to $0.$ For the same reason we get:
		\begin{equation*}
			\lim\limits_{\epsilon\rightarrow 0}	\int_{\{u>0\}} |u|^{q-1}\langle \Phi. \nabla u \rangle=\lim\limits_{\epsilon\rightarrow 0}\int_{\{u>0\}} |u|^{q-1}\Psi_{\epsilon}(u)\langle \varphi. \nabla u \rangle=0
		\end{equation*}
		\begin{eqnarray*}
			\int_{\{u>0\}} div \Phi |\nabla u|^p&=&\int_{\{u>0\}}\Psi_{\epsilon}(u)div \varphi |\nabla u|^p+ \Psi_{\epsilon}'(u)\langle \varphi. \nabla u \rangle|\nabla u|^p\\
			&=&\int_{\{u>0\}}\Psi_{\epsilon}(u)div \varphi |\nabla u|^p-\frac{1}{\epsilon} \int_{\{0<u<\epsilon\}}\langle \varphi. \nabla u \rangle|\nabla u|^p.
		\end{eqnarray*}
		Also we have 
		\begin{eqnarray*}
			\int_{\{u>0\}} div \Phi |\nabla u|^2&=&\int_{\{u>0\}}\Psi_{\epsilon}(u)div \varphi |\nabla u|^2+ \Psi_{\epsilon}'(u)\langle \varphi. \nabla u \rangle|\nabla u|^2\\
			&=&\int_{\{u>0\}}\Psi_{\epsilon}(u)div \varphi |\nabla u|^p-\frac{1}{\epsilon} \int_{\{0<u<\epsilon\}}\langle \varphi. \nabla u \rangle|\nabla u|^2
		\end{eqnarray*}
		Using $\nabla u \in L^p$, the first terms goes to $0$. Finally we also have,
		\begin{equation*}
			\int_{D} \langle D\Phi\nabla u.\nabla u \rangle=\int_{\{u>0\}}\Psi_{\epsilon}(u)\langle D\varphi\nabla u\nabla u \rangle-\frac{1}{\epsilon}\int_{\{0<u<\epsilon\}}\langle \varphi. \nabla u \rangle|\nabla u|^2
		\end{equation*}
		\begin{equation*}
			\int_{D} |\nabla u|^{p-2}\langle D\Phi\nabla u.\nabla u \rangle=\int_{\{u>0\}}\Psi_{\epsilon}(u)|\nabla u|^{p-2}\langle D\varphi\nabla u\nabla u \rangle-\frac{1}{\epsilon}\int_{\{0<u<\epsilon\}}\langle \varphi. \nabla u \rangle|\nabla u|^p
		\end{equation*}
		and the first the first term goes to $0$. By writing Euler-Lagrange’s (\ref{euler1}) equation and letting " goes to 0, we get:
		\begin{eqnarray*}
			& &\lim\limits_{\epsilon\rightarrow 0}		\frac{1}{\epsilon}\int_{\{0<u<\epsilon\}}\langle \varphi. \nabla u \rangle\left(\frac{1-p}{p}|\nabla u|^p-\frac{1}{2}|\nabla u|^2
			+\Lambda\right) =0
		\end{eqnarray*}
	\end{proof}
	\begin{proof}[Proof of Theorem \ref{theon2.4}]
		Let $B(x,r)\subset D$ and $\varphi\in C^{\infty}_0(B(x,r)).$ For almost every $s > 0$ the boundary of $\{u>s\}$ is regular $(C^1)$, since on the open set $\{u>0\}$ we have $-\Delta_pu-\Delta u=\lambda_c|u|^{q-2}u$ so that $u$ is $C^1$ and we can use Sard’s
		lemma  (see Lemma 13.15 in \cite{Maggi2012}), which implies that $|\nabla u|>0$ on $\{u=s\}$ for almost every $s$. We can now write, using co-area formula (see 3.4.3 in \cite{Evans}), and Gauss formula:
		\begin{eqnarray*}	
			\frac{1}{\epsilon}\int_{\{0<u<\epsilon\}}\langle \varphi. \nabla u \rangle&=&\frac{1}{\epsilon}\int_{0}^{\epsilon}ds\int_{\{u=s\}}\langle \varphi. \frac{\nabla u}{|\nabla u|} \rangle\\
			&=&\frac{1}{\epsilon}\int_{0}^{\epsilon}ds\int_{\{u=s\}}\langle \varphi.\nu_s \rangle\\
			&=& \frac{1}{\epsilon}\int_{0}^{\epsilon}ds\int_{\{u>s\}}div \varphi\\
		\end{eqnarray*}
		(here $\nu_s$ is the outward normal to $\{u>s\}).$ We deduce,
		\begin{equation}\label{eqn18}
			\lim\limits_{\epsilon\rightarrow 0}		\frac{1}{\epsilon}\int_{\{0<u<\epsilon\}}\langle \varphi. \nabla u \rangle =\int_{\{u>0\}}div \varphi.	
		\end{equation}
		There exist $s_0<r$ such that $\varphi\in C^{\infty}_0(B(x,s_0)).$ If we suppose that $\|\varphi\|_{\infty}\leq 1,$ and using the fact that $u\in W^{1,\infty}_{loc} (D)$ then
		\begin{eqnarray}\label{eqn19}
			\frac{1}{\epsilon}\int_{\{0<u<\epsilon\}}\langle \varphi. \nabla u \rangle\left(\frac{p-1}{p}|\nabla u|^p+\frac{1}{2}|\nabla u|^2\right) \leq \frac{p-1}{p\epsilon}\|\nabla u\|_{L^{\infty}(B(x,s_0))}\int_{\{0<u<\epsilon\}\cap B(x,s_0)}\left(|\nabla u|^p+|\nabla u|^2\right) \nonumber\\
		\end{eqnarray}
		for every $s_0<s<r.$ But we saw in the proof of Proposition \ref{Lem5.3} (see (\ref{eqn17})) that
		\begin{equation}
			\lim\limits_{n\rightarrow \infty}n\left(|\nabla u|^p+|\nabla u|^2\right)\chi_{\{0<u<\frac{1}{n}\}}=\mu_1,
		\end{equation}
		weakly in the sense of Radon’s measure ( see 1.9 in \cite{Evans}). For almost every $s < r,$
		\begin{equation*}
			\mu_1 (\partial B(x,s) )=0.
		\end{equation*}
		Let take a $ s > s_0$, we get:
		\begin{eqnarray}
			\lim\limits_{n\rightarrow \infty}n\int_{\{0<u<\frac{1}{n}\}\cap B(x,s)}\left(|\nabla u|^p+|\nabla u|^2\right)= \mu_1(B(x,s))
		\end{eqnarray}
		From Lemma \ref{Lemn2.5}, (\ref{eqn18}) and (\ref{eqn19}) with $\epsilon=\frac{1}{n}$ and $n\rightarrow\infty$
		\begin{eqnarray*}
			\Lambda\int_{\Omega_u\cap B(x,r)} div \varphi&\leq&	\frac{p-1}{p}\|\nabla u\|_{L^{\infty}(B(x,r))} \mu_1(B(x,s))\\
			&\leq&	\frac{p-1}{p}\|\nabla u\|_{L^{\infty}(B(x,r))} \mu_1(B(x,r))\\
		\end{eqnarray*}
		That is, by taking the supremum over $\varphi$:
		\begin{equation*}
			P(\Omega_u, B(x,r))\leq\frac{p-1}{p\Lambda}\|\nabla u\|_{L^{\infty}(B(x,r))} \mu_1(B(x,r))
		\end{equation*}
		To prove that $\Omega_u$ has locally finite perimeter in $D$, we use that $\mu_1$ is finite on the bounded set $D$. Using Proposition \ref{propn2.3}, we get for $B(x,r)\subset D$ and $r$ small enough:
		\begin{equation*}
			P(\Omega_u, B(x,r))\leq C r^{d-1}.
		\end{equation*}
	\end{proof}
	\begin{prop}\label{pronm}
		Let $u$ be a solution (\ref{equation 6}). Let $x\in\partial\Omega_u.$ Then there exist $C_1$, $C_2$ and  $r_0$ such that, for every $B(x,r)\subset D$ and $r\leq r_0,$ we have
		\begin{equation}\label{reduit1}
			0<C_1\leq\frac{|\Omega_u\cap B(x,r)|}{| B(x,r)|}\leq C_2<1,
		\end{equation}
	and 
	\begin{equation}\label{reduit2}
		0<\frac{|\Omega_u^c\cap B(x,r)|}{| B(x,r)|}<1
	\end{equation}
where $\Omega_u^c=D\backslash \Omega_{u}$	
		Moreover we have:
		\begin{eqnarray*}
			\mathcal{H}^{d-1}((\partial\Omega_u\backslash\partial^*\Omega_u)\cap D)=0.
		\end{eqnarray*}
	\end{prop}
	\begin{Lem}\label{lema2.5}
		Let $\omega$ be an open subset of $D$, and let $u$ be a solution of (\ref{equation 6}). If $|\Omega_u\cap \omega|=|\omega|$ then
		\begin{equation*}
			-\Delta_p u-\Delta u=\lambda_c|u|^{q-2}u \ \mbox{in}\ \omega,
		\end{equation*}
		and therefore $\omega\subset \Omega_u.$	
	\end{Lem}
	\begin{proof}[Proof of Lemma \ref{lema2.5}]
		Let $\omega$ be an open subset of $D$, and let $u$ be a solution of (\ref{equation 6}).
		Let's reason by the absurd and assume that 	$|\Omega_u\cap \omega|=|\omega|$ and	
		\begin{equation*}
			-\Delta_p u-\Delta u\neq\lambda_c|u|^{q-2}u \ \mbox{in}\ \omega.
		\end{equation*}	
		It is known that 		
		\begin{equation*}
			-\Delta_p u-\Delta u=\lambda_c|u|^{q-2}u \ \mbox{in}\ \Omega_u.
		\end{equation*}	
		so we have 		
		\begin{equation*}
			-\Delta_p u-\Delta u=\lambda_c|u|^{q-2}u \ \mbox{in}\ \Omega_u\cap \omega
		\end{equation*}			
		and 		
		\begin{equation*}
			-\Delta_p u-\Delta u\neq\lambda_c|u|^{q-2}u \ \mbox{in}\ \omega\backslash(\Omega_u\cap \omega).
		\end{equation*}	
		This implies that $|\omega\backslash(\Omega_u\cap \omega)|\neq 0.$
		Furthermore, we $\omega=(\Omega_u\cap \omega)\cup(\omega\backslash(\Omega_u\cap \omega))$ and 
		$|\omega|=|\Omega_u\cap \omega|+ |\omega\backslash(\Omega_u\cap \omega|$ and as $|\Omega_u\cap \omega|=|\omega|$ then  $|\omega\backslash(\Omega_u\cap \omega)|= 0.$			
	\end{proof}
	\begin{proof}[Proof of Proposition \ref{pronm}]
		By theorem \ref{theon2.4}, $\Omega_u$ is locally finite perimeter in $D$. So using proposition 12.19 in \cite{Maggi2012} we have 
		\begin{equation*}
			Spt \mu_{\Omega_u}	=\{x\in \mathbb{R}^d: 0<|\Omega_u\cap B(x,r)|< \omega_dr^d, \forall r>0\}\subset \partial\Omega_u,
		\end{equation*}
		where $\mu_{\Omega_u}$ is the Gauss–Green measure of $\Omega_u$.\\
		Let $x\in \partial\Omega_u$, such that $x\notin Spt \mu_{\Omega_u}.$ This implies $|\Omega_u\cap B(x,r)|\in \left\lbrace 0;\omega_dr^d\right\rbrace $ for $r$ small.
		If $|\Omega_u\cap B(x,r)|=\omega_dr^d=|B(x,r)|$, applying the lemma \ref{lema2.5} to $\omega=B(x,r)$, we would get $\Omega_u\cap B(x,r)=B(x,r)$, which is impossible since $B(x,r)$ is centered on $\partial\Omega_u.$\\
		Since for $D$ connected, $\Omega_u$ is an open set and then $|\Omega_u\cap B(x,r)|>0$. So for all $x\in \partial\Omega_u$ we have (\ref{reduit1}).\\
	
	Furthermore, let $x\in \partial\Omega_u$ since $\Omega_u^c$ is a closed set and $ \Omega_u^c\cap B(x,r)\subset B(x,r) $ then for $r$ small we have 
	\begin{equation*}
		0<\frac{|\Omega_u^c\cap B(x,r)|}{| B(x,r)|}\leq 1.
	\end{equation*}
Now let at assume that $|\Omega_u^c\cap B(x,r)|=	| B(x,r)|$, which would imply that $|\Omega_u\cap B(x,r)|=0$. Contradiction because $\Omega_u$ is an open set.  Thus we have (\ref{reduit2}).\\
	
		This completes the proof of the first part. The second part comes from the theory of sets with finite perimeter (see 5.8
		in \cite{Evans}).
	\end{proof}

\end{document}